\documentclass[
    11pt,
    english,
    a4paper,
    final,
]{scrartcl}
\usepackage{amsmath}
\usepackage{color} 
\usepackage{amssymb}
\usepackage{amsthm}
\usepackage{accents}
\usepackage{float}
\usepackage{caption}
\usepackage[ruled,vlined,procnumbered]{algorithm2e}
\usepackage[english]{babel}
\usepackage{bbm}
\usepackage{chngcntr}
\usepackage{enumitem}
\usepackage{listings}
\usepackage{lmodern} 
\usepackage{mathtools}
\usepackage{setspace,graphicx,tikz,tabularx} 
\usepackage{csquotes}
\usepackage[textwidth=16cm,textheight=23.5cm]{geometry}
\usepackage{soul}

\usepackage[unicode=true,hypertexnames=false]{hyperref}
\usepackage[capitalise,noabbrev,nameinlink,sort]{cleveref}
\let\etoolboxforlistloop\forlistloop 
\usepackage{autonum}
\let\forlistloop\etoolboxforlistloop 
\cslet{blx@noerroretextools}\empty 
\usepackage[style=numeric,url=false,sortcites=true,isbn=false,doi=false,eprint=true,date=year,maxbibnames=9]{biblatex}

\DeclareNameAlias{default}{family-given}
\counterwithin{equation}{section}

\DeclareMathOperator*{\argmax}{arg\,max}
\DeclareMathOperator*{\argmin}{arg\,min}

\renewcommand{\c}[1]{\mathcal{#1}}

\renewcommand{\b}[1]{\mathbb{#1}}

\def \mra{\mathrm{a}} 
\def \mrJ{\mathrm{J}}

\newcommand{\N}{\mathbb{N}}
\newcommand{\R}{\mathbb{R}}
\renewcommand{\P}{\mathbb{P}}
\DeclareMathOperator{\E}{\mathbb{E}}

\newcommand{\1}{\mathbbm{1}}

\newtheorem{theorem}{Theorem}[section]

\newtheorem{remark}[theorem]{Remark}

\newtheorem{assumption}{Assumption}
\crefname{assumption}{Assumption}{Assumptions}

\renewcommand{\theassumption}{\Alph{assumption}}
\newlist{assumptionenum}{enumerate}{1} 
\setlist[assumptionenum]{label=(\theassumption\arabic*)}
\crefalias{assumptionenumi}{assumption}

\let\eqref\labelcref
\makeatletter
\autonum@generatePatchedReferenceCSL{eqref}
\makeatother

\begin{document}

\title{
\textnormal{
Control randomisation approach for policy gradient and application to reinforcement learning in optimal switching
}
}

\author{Robert DENKERT\footnote{Humboldt University of Berlin, Department of Mathematics, denkertr at hu-berlin.de; This author gratefully acknowledges financial support by the Deutsche Forschungsgemeinschaft (DFG, German Research Foundation) -- Project-ID 410208580 -- IRTG2544 (”Stochastic Analysis in Interaction”).} \quad  
Huy\^en PHAM\footnote{LPSM,  Universit\'e Paris Cité and Sorbonne Universit\'e, pham at lpsm.paris; This author  is supported by  the BNP-PAR Chair ``Futures of Quantitative Finance", 
and by FiME, Laboratoire de Finance des March\'es de l'Energie, and the ``Finance and Sustainable Development'' EDF - CACIB Chair} \quad 
Xavier WARIN \footnote{  EDF R\&D and FiME, Laboratoire de Finance des March\'es de l'Energie.  }  
}

\maketitle

\begin{abstract}
\noindent \textbf{Abstract}. 
We propose a comprehensive framework for policy gradient methods tailored to continuous time reinforcement learning.  This is based on the connection between stochastic control problems and  randomised problems, enabling applications across various classes of Markovian  continuous time control problems, beyond diffusion models, including e.g. regular, impulse and optimal stopping/switching problems. By utilizing  change of measure in the control randomisation technique, we derive a new policy gradient representation for these randomised problems, featuring parametrised intensity policies. 
We further develop actor-critic algorithms specifically designed to address  general Markovian stochastic control issues. 
Our framework is demonstrated through its application to optimal switching problems, with two numerical case studies in the energy sector focusing  on real options. 
\end{abstract}

\vspace{5mm}


\newpage

\section{Introduction}

The theory of reinforcement learning for continuous time stochastic control has advanced significantly, beginning with the foundational work \cite{wanzarzho20}, and continuing with \cite{jia_policy_2021}, \cite{jia_policy_2021-1} who developed policy gradient methods and actor-critic algorithms, and \cite{jiazhou23} for $q$-learning.  These studies primarily focus on 
regular controls within diffusion processes, employing the Feynman-Kac formula and the partial differential equations (PDE)  representation of the value function to derive gradients of the performance value function with respect to the parameters of the stochastic policy. 

Our research aims to expand the application of these methods beyond diffusion models to a broader range of Markovian control problems, including singular, impulse, and optimal stopping and switching problems. 
To achieve this, we propose a unified framework with a general reformulation in terms of  Markovian randomised problems. This approach to stochastic control problem is commonly referred to as \textit{control randomisation}, initially introduced in \cite{bouchard_stochastic_2009} for optimal switching problems, and further developed in \cite{khamaphazha10} for impulse control, in \cite{kharroubi_feynmankac_2015} for regular controls, and in \cite{fuhpha15} for general non-Markovian stochastic control problems.  The basic idea is to replace the control process $(\alpha_t)_t$ valued in $A$ by a random (uncontrolled) point  process $(I_t)_t$ with marks in $A$, formulate an auxiliary control problem where the intensity distribution of $I$ is controlled, called randomized problem, and show that the value functions of the two problems coincide. The key feature of the randomised problems is its formulation in terms of a family of dominated probability measures under which the optimization is performed. 

Utilizing the change of measure in these randomised settings, we derive a gradient representation of the value function with respect to parametrised intensity policies directly, without reliance on PDEs. This framework not only incorporates Poisson discretisation as per the randomization method but also accommodates standard fixed discretisations for continuous-time problems. Using this policy gradient, we design an Actor-Critic algorithm to alternately learn the value function and the optimal intensity policy. Notably, the gradient structure relies solely on the state at action points, circumventing the need for further discretisation during implementation.

We demonstrate the applicability of our results in a model-free setting, learning optimal control and value functions through empirical observations and samples. This methodology is applied specifically to optimal switching problems but is adaptable to a wide variety of continuous stochastic control scenarios. We provide numerical examples from real options in the energy markets to illustrate these concepts. 

The remainder of the paper is structured as follows: In \cref{section unified approach}, we detail the Markovian randomised problem and develop a corresponding policy gradient method. \cref{section application to stochastic control problems} applies this methodology to a diverse array of continuous time control problems, and \cref{section numerical experiments} presents and evaluates numerical experiments within the context of optimal switching problems.

\section{Policy gradient method for Markovian randomised problems}\label{section unified approach}

In this section, we will consider a general class of Markovian randomised control problems in continuous time. 
Control randomisation method can be seen as a unified approach to a large class of control problems in continuous time, including optimal stopping, switching and impulse control problems, as we will see in \cref{section application to stochastic control problems}.
We will derive first a general policy gradient representation  and then from this, an Actor-Critic algorithm to tackle this class of problems.

\subsection{Theory background}\label{subsection theory background}

\subsubsection{Randomised control problem setup}

Let $(\Omega,\c F,\P)$ be a probability space on which we consider a simple random counting measure $\nu$ on 
$(0,\infty)\times A$ with $A$ some Polish space, such that $\E[\nu((0,T]\times A)]$ $<$ $\infty$, and   associated to the marked point process $(\tau_n,\mra_n)_n$ and the pure jump $A$-valued process $I$ with dynamics
\begin{align} \label{dynI}
d I_s & = \; \int_A (e - I_{s-}) \nu(ds,de), \quad s \leq T.     
\end{align}
We denote by $I^{t,a}$ the jump process starting from $a$ $\in$ $A$, at time $t$ $\in$ $[0,T]$, that is $I_t^{t,a}=a$, and following the dynamics \eqref{dynI} for $t$ $\leq$ $s$ $\leq$ $T$. 
We consider a state process $X$ valued on $\R^d$ s.t. the pair $(X,I)$ is Markov, and $X$ only jumps at the times given by $\nu$. 
An example includes the case where $X$ is driven by a SDE in the form
\begin{align} \label{exdiffX}
d X_s &= \; b(s,X_s,I_s) d s + \sigma(s,X_s,I_s) dW_s + \int_A \gamma(s,X_{s-},I_{s-},e) \nu(ds,de),   
\end{align}
with $W$ a Brownian motion. 
We denote by $X^{t,x,a}$ the state process $X$ that starts from $x$ at time $t$, that is $X_{t}^{t,x,a}$ $=$ $x$, and s.t. $(X^{t,x,a},I^{t,a})$ is Markov, and we assume the estimate 
\begin{align}
\E \Big[ \sup_{t\leq s\leq T} |X_s^{t,x,a}|^p \Big]  & \leq \; C(1 +  |x|^p), \quad \quad \forall x \in \R^d,  
\end{align}
for some positive constant $C$ and $p$ $\in$ $[1,\infty)$.  This estimate is satisfied for $X$ as in \eqref{exdiffX} under standard Lipschitz and linear growth conditions on $b,\sigma,\gamma$.

By \autocite[Theorems 2.1, 2.3, 3.4]{jacod_multivariate_1975}, there exists a unique (up to a $\P$-null set) predictable random measure $\hat\nu$ with $\hat\nu(\{s\}\times A) \leq 1$ for all $s\in (0,T]$ such that for every $\c P(\b F^{X,\nu})\otimes \c B(A)$-measurable random field $H\geq 0$, where $\c P(\b F^{X,\nu})$ denotes the predictable $\sigma$-algebra of $\b F^{X,\nu}$, it holds that
\[
\E\bigg[\int_0^T \int_A H(s,e) \nu(ds,de)\bigg] = \E\bigg[ \int_0^T \int_A H(s,e) \hat\nu(ds,de)\bigg],
\]
called the \emph{predictable projection or compensator of $\nu$}, which is uniquely characterising $\nu$.
Guided by the approach of the randomisation method, we will now optimise over the set of (in a suitable sense) \grqq intensities\grqq{} of the process $I$. To this end, we note that for every $\b F^{X,\nu}\otimes\c B(A)$-predictable, essentially bounded process $\lambda$ satisfying
\begin{enumerate}[label=(\roman*)]
    \item $\int_A \lambda_s(e) \hat\nu(\{s\}, de) \leq 1$ for all $s\in (0,T]$,
    \item for all $s\in (0,T]$ such that $\hat\nu(\{s\}\times A) = 1$, it also holds that $\int_A \lambda_s(e) \hat\nu(\{s\}, de) = 1$,
\end{enumerate}
we can construct a tilted probability measure $\P^\lambda\ll \P$ such that $\nu$ is a random point measure with the predictable projection $\lambda_s(e) \hat\nu(ds,de)$ under $\P^\lambda$. This is achieved through Girsanov's theorem, as outlined in e.g. \autocite[
Theorem 4.5]{jacod_multivariate_1975}, by defining $\P^\lambda$ via its density process
\begin{align}
   Z_s^\lambda \; := \;  \frac{d\P^\lambda}{d\P}\bigg|_{\c F^{X,\nu}_s}
    & = \;  \prod_{t\in (0,s], 0 < \hat\nu(\{t\}\times A) < 1,  \nu(\{t\}\times A) = 0}   \frac{1 - \int_A \lambda_{t}(e) \hat\nu(\{t\}, de)}{1 - 
\hat\nu(\{t\}\times A)}\\
    &\qquad\cdot\exp\bigg(
    \int_{(0,s]}\int_A \log\lambda_t(e) \nu(dt,de)
    - \int_{(0,s]}\int_A (\lambda_t(e) - 1) \hat\nu^c(dt,de)\bigg),
    \label{eq girsanov formula}
\end{align}
for $s$ $\in$ $(0,T]$, where $\hat\nu^c(ds,de)\coloneqq \1_{\{\hat\nu(\{s\}\times A) = 0\}} \hat\nu(ds, de)$.

Notice that we do not assume necessarily that the compensator is absolutely continuous w.r.t. the Lebesgue measure $ds$, in order to take into account the possibility of jumps at deterministic times, hence to embed the case of stochastic control on discrete time, i.e. Markov decision process.

By the Markovian structure of our problem, we now define the set of admissible control $\c V$ as all such processes $\lambda$ satisfying the above conditions while being of the form
\[
\lambda_s(e) = \lambda(e | s, X^{}_{s-}, I^{}_{s-}),\qquad s\leq T,
\]
for some bounded deterministic function $\lambda$ on $A\times [0,T]\times\R^d\times A$.
For the ease of arguments, we furthermore require that all $\lambda\in\c V$ satisfy the following conditions which ensure that also $\P\ll\P^\lambda$ and thus $\P^\lambda\sim\P$,
\footnote{Since every $\lambda$ with $\P^\lambda\ll\P$ can be approximated by $(\lambda^n)_n\subseteq\c V$, this additional assumption also does not change the value function. Similar arguments are standard for randomised control problems.}
\begin{enumerate}[label=(\roman*)]
    \item $\lambda$ is bounded away from 0, that is $\inf_{(s,x,a,e)} \lambda(e|s,x,a) > 0$,
    \item there exists a constant $C < 1$ such that for all $s\in (0,T]$, when $\hat\nu(\{s\}\times A) < 1$, then it also holds that $\int_A \lambda(e | s, X^{}_{s-}, I^{}_{s-}) \hat\nu(\{s\},de) \leq C < 1$.
\end{enumerate}

Note that for each such $\lambda\in\c V$, the process $(X,I)$ will still be Markovian under $\P^\lambda\ll \P$, and we have the estimate
\begin{align}
\E^\lambda \Big[ \sup_{t\leq s\leq T} |X_s^{t,x,a}|^p \Big]  & \leq \; C_\lambda(1 +  |x|^p), \quad \quad \forall x \in \R^d,  
\end{align}
where $\E^\lambda$ denotes the expectation under $\P^\lambda$. 
In general, our objective is now to optimise the reward functional: 
\begin{align}
    J(t,x,a,\lambda) & \coloneqq \;  \E^\lambda\bigg[g(X^{t,x,a}_T,I^{t,a}_T) + \int_t^T f(s,X^{t,x,a}_s,I^{t,a}_s) ds - \int_{(t,T]}\int_A c(s,X^{t,x,a}_{s-},I^{t,a}_{s-},e) \nu(ds,de) \bigg],
\end{align}
for $(t,x,a)$ $\in$ $[0,T]\times\R^d\times A$, where the reward functions $f$, $g$ and the cost function $c$ are assumed to satisfy the polynomial growth condition
\begin{align}
|f(t,x,a)| + |g(x,a)| + |c(t,x,a,e)|    & \leq C (1 + |x|^p), 
\end{align}
for all $t$ $\in$ $[0,T]$, $x$ $\in$ $\R^d$, $a,e$ $\in$ $A$.  Notice that the reward functional $J$ then also satisfies the 
the polynomial growth condition
\begin{align}
|J(t,x,a,\lambda)|  & \leq C_\lambda (1 + |x|^p).  
\end{align}

\begin{remark}\label{remark markov property of reward process}
From the definition of the reward functional $J$ and the Markov property of $(X,I)$, we have the martingale property under $\P^\lambda$, $\lambda$ $\in$ $\c V$,  of the process 
\begin{align}
J(s,X_s,I_s,\lambda) + \int_0^s f(r,X_r,I_r) dr - \int_{(0,s]}\int_A c(r,X_{r-},I_{r-},e) \nu(dr,de),    \quad 0 \leq s \leq T.  
\end{align}
\end{remark}

\subsubsection{Policy gradient representation}

The policy gradient method aims to optimize the expected reward $J$ by exploring a parameterized family $(\lambda^\theta)_{\theta\in\Theta}\subseteq \c V$. 
This family is chosen to be sufficiently dense in $\c V$, meaning that
\[
\sup_{\lambda \in \c V} J(t,x,a,\lambda) = \sup_{\theta\in\Theta} J(t,x,a,\theta),
\]
where we denote by $J(t,x,a,\theta) \coloneqq J(t,x,a,\lambda^\theta)$ with a slight abuse of notation.  
The optimization process then involves computing the gradient of $J^\theta$ with respect to the parameter $\theta$, allowing for updates to the policy parameters -- typically done through methods like gradient descent -- to maximize the overall reward.

Our aim in this section is now  to derive an explicit formula for the gradient $\triangledown_\theta J^\theta(t,x,a,\theta)$. While the approach by \autocite{jia_policy_2021} is based on the Feynman-Kac formula for $J^\theta$, we will instead use the Girsanov formula \eqref{eq girsanov formula}. 
The advantage is that we do not need to assume or impose conditions for ensuring regularity on the functional $J$ 
for deriving the partial differential equations that it satisfies in the continuous-time framework. 
This is crucial since the function $J$ may be discontinuous in time in the case where $\hat\nu$ admits atoms in time, and then PDE method cannot be applied.  

We shall assume that for all $(t,x,a,e)$ $\in$ $[0,T]\times\R^d\times A\times A$, the map $\theta$ $\in$ $\Theta$ $\mapsto$ 
$\lambda^\theta(e|t,x,a)$ is differentiable with a derivative satisfying the growth condition: for each $\theta\in\Theta$, 
there exists some positive constant $C_\theta$ s.t. 
\begin{align}
\int_{(t,T]} \int_A \big| \nabla_\theta \lambda^\theta(e|s,x,a) \big| \hat\nu(ds,de) & \leq \; C_\theta(1 + |x|),     
\quad (t,x) \in [0,T]\times\R^d.
\end{align}

\begin{theorem}
We have 
\begin{align}
    \nabla_\theta J(t,x,a,\theta)
    &=\E^\theta\bigg[\int_{(t,T]}\int_A \triangledown_\theta (\log\lambda^\theta)(e|s,X^{t,x,a}_{s-},I^{t,a}_{s-})\\
    &\qquad \quad  \cdot \Big(J(s,X^{t,x,a}_s,e,\theta) - J(s,X^{t,x,a}_{s-},I^{t,a}_{s-},\theta) - c(s,X^{t,x,a}_{s-},I^{t,a}_{s-},e)\Big) \nu(ds,de)\bigg],\label{eq policy gradient equation}
\end{align}
for $(t,x,a)$ $\in$ $[0,T]\times\R^d\times A$, where $\E^\theta$ denotes the expectation under 
$\P^\theta$ $=$ $\P^{\lambda^\theta}$. 
\end{theorem}
\begin{proof}
From Bayes formula with \eqref{eq girsanov formula}, the reward functional is formulated in term of the reference probability measure $\P$ instead of $\P^\theta\coloneqq \P^{\lambda^\theta}$ as follows for $\theta\in\Theta$,
\begin{align}
    & \quad J(t,x,a,\theta)\\
    = & \;  \E^\theta\bigg[g(X^{t,x,a}_T,I^{t,a}_T) + \int_t^T f(s,X^{t,x,a}_s,I^{t,a}_s) ds - \int_{(t,T]}\int_A c(s,X^{t,x,a}_{s-},I^{t,a}_{s-},e) \nu(ds,de) \bigg]\\
    = & \;  \E\bigg[ Z_{T}^{t,x,a,\theta}  
    \bigg(g(X^{t,x,a}_T,I^{t,a}_T) + \int_t^T f(s,X^{t,x,a}_s,I^{t,a}_s) ds - \int_{(t,T]}\int_A c(s,X^{t,x,a}_{s-},I^{t,a}_{s-},e) \nu(ds,de)\bigg) \bigg],
\end{align}
where 
\begin{align}
Z_{T}^{t,x,a,\theta} & = \;  
\exp \bigg( \int_{(t,T]}\int_A \log\lambda^\theta(e|s,X_{s-}^{t,x,a},I_{s-}^{t,a}) \nu(ds,de)  
- \int_{(t,T]}\int_A (\lambda^\theta(e|s,X_{s-}^{t,x,a},I_{s-}^{t,a}) - 1) \hat\nu^c(ds,de) \\
& \quad \quad \quad  + \sum_{s\in (t,T], 0 < \hat\nu(\{s\}\times A) < 1} \1_{\{\nu(\{s\}\times A) = 0\}}  
 \log\Big( \frac{1 - \int_A \lambda^\theta(e|s,X^{t,x,a}_{s-},I^{t,a}_{s-}) \hat\nu(\{s\}, de)}
 {1 - \hat\nu(\{s\}\times A)}   \Big) \bigg).  
\end{align} 
By differentiating this relation w.r.t. $\theta$, and writing $\nabla_\theta Z_T^{t,x,a,\theta}$ $=$ 
$Z_T^{t,x,a,\theta} L_T^{t,x,a,\theta}$ with
\begin{align}
L_T^{t,x,a,\theta} &= \;  
\int_{(t,T]}\int_A \triangledown_\theta (\log\lambda^\theta)(e|s,X^{t,x,a}_{s-},I^{t,a}_{s-}) \nu(ds,de) - \int_{(t,T]} \int_A \triangledown_\theta \lambda^\theta(e|s,X^{t,x,a}_{s-},I^{t,a}_{s-}) \hat\nu^c(ds,de)  \\
& \quad \quad  - \; \sum_{s\in (t,T], 0 < \hat\nu(\{s\}\times A) < 1} \1_{\{\nu(\{s\}\times A) = 0\}} 
\frac{\int_A \triangledown_\theta \lambda^\theta(e|s,X^{t,x,a}_{s-},I^{t,a}_{s-}) \hat\nu(\{s\}, de)}{1-\int_A \lambda^\theta(e|s,X^{t,x,a}_{s-},I^{t,a}_{s-}) \hat\nu(\{s\}, de)},  
\end{align}
we get 
\begin{align}
& \quad  \nabla_\theta J(t,x,a,\theta) \\
= & \E^\theta\bigg[ L_T^{t,x,a,\theta} \Big( g(X^{t,x,a}_T,I^{t,a}_T) + \int_t^T f(s,X^{t,x,a}_s,I^{t,a}_s) ds - \int_{(t,T]}\int_A c(s,X^{t,x,a}_{s-},I^{t,a}_{s-},e) \nu(ds,de) \Big)  \bigg]
\end{align}
To simplify this expression, we will use that due the Markovian structure of our problem, the process $M^{t,x,a,\theta}$ given by
\[
M_s^{t,x,a,\theta} \coloneqq J(s,X^{t,x,a}_s,I^{t,a}_s,\theta) + \int_t^s f(r,X^{t,x,a}_r,I^{t,a}_r) dr - \int_{(t,s]}\int_A c(r,X^{t,x,a}_{r-},I^{t,a}_{r-},e)\nu(dr,de),\quad s\in [t,T],
\]
is a $\P^\theta$-martingale, see also \cref{remark markov property of reward process}. We start by noting that $J(T,X^{t,x,a}_T,I^{t,a}_T,\theta) = g(X^{t,x,a}_T,I^{t,a}_T)$, which allows us to write $\nabla_\theta J(t,x,a,\theta)$ using $M^{t,x,a,\theta}$ as follows
\begin{align}
\nabla_\theta J(t,x,a,\theta) &= \E^\theta[L^{t,x,a,\theta}_T M^{t,x,a,\theta}_T]\\
&= \E^\theta\bigg[
\int_{(t,T]}\int_A \triangledown_\theta (\log\lambda^\theta)(e|s,X^{t,x,a}_{s-},I^{t,a}_{s-}) M^{t,x,a,\theta}_T \nu(ds,de)\\
&\qquad \;  - \;  \int_{(t,T]} \int_A \triangledown_\theta \lambda^\theta(e|s,X^{t,x,a}_{s-},I^{t,a}_{s-}) M^{t,x,a,\theta}_T \hat\nu^c(ds,de)  \\
& \quad \quad \;  -  \;  \sum_{s\in (t,T], 0 < \hat\nu(\{s\}\times A) < 1} \1_{\{\nu(\{s\}\times A) = 0\}} 
\frac{\int_A \triangledown_\theta \lambda^\theta(e|s,X^{t,x,a}_{s-},I^{t,a}_{s-}) \hat\nu(\{s\}, de)}{1-\int_A \lambda^\theta(e|s,X^{t,x,a}_{s-},I^{t,a}_{s-}) \hat\nu(\{s\}, de)} M^{t,x,a,\theta}_T
\bigg].
\end{align}
Now using the $\P^\theta$-martingale property of $M^{t,x,a,\theta}$, we obtain
\begin{align}
\nabla_\theta J(t,x,a,\theta)
&= \E^\theta\bigg[
\int_{(t,T]}\int_A \triangledown_\theta (\log\lambda^\theta)(e|s,X^{t,x,a}_{s-},I^{t,a}_{s-}) M^{t,x,a,\theta}_s \nu(ds,de)\\
&\qquad \;  - \;  \int_{(t,T]} \int_A \triangledown_\theta \lambda^\theta(e|s,X^{t,x,a}_{s-},I^{t,a}_{s-}) M^{t,x,a,\theta}_s \hat\nu^c(ds,de)  \\
& \quad \quad \;  -  \;  \sum_{s\in (t,T], 0 < \hat\nu(\{s\}\times A) < 1} \1_{\{\nu(\{s\}\times A) = 0\}} 
\frac{\int_A \triangledown_\theta \lambda^\theta(e|s,X^{t,x,a}_{s-},I^{t,a}_{s-}) \hat\nu(\{s\}, de)}{1-\int_A \lambda^\theta(e|s,X^{t,x,a}_{s-},I^{t,a}_{s-}) \hat\nu(\{s\}, de)} M^{t,x,a,\theta}_s
\bigg].
\label{eq dtheta J step 2}
\end{align}

To simplify the notation in the following arguments, let us introduce the predictable process
\[
N_s^{t,x,a,\theta} \coloneqq J(s,X^{t,x,a}_{s-},I^{t,a}_{s-},\theta) + \int_t^s f(r,X^{t,x,a}_r,I^{t,a}_r) dr - \int_{(t,s)}\int_A c(r,X^{t,x,a}_{r-},I^{t,a}_{r-},e)\nu(dr,de),\quad s\in [t,T]. 
\]
We note that $\nu$ and thus also $I^{t,a}$ and $X^{t,x,a}$ have $\P$-a.s.\@ only finitely many discontinuities on $[t,T]$. This implies that $\{s \in [t,T] | M^{t,x,a,\theta}_s \not= N^{t,x,a,\theta}_s\}$ is $\P$-a.s.\@ countable and thus $M^{t,x,a,\theta}_s = N^{t,x,a,\theta}_s$, $\P\otimes\hat\nu^c(\cdot,A)$-a.s., which allows us to rewrite the second term in \eqref{eq dtheta J step 2} as
\begin{align}
    &\E^\theta\bigg[\int_{(t,T]}\int_A \triangledown_\theta \lambda^\theta(e|s,X^{t,x,a}_{s-},I^{t,a}_{s-}) M^{t,x,a,\theta}_s \hat\nu^c(ds,de)\bigg]\\
    &=\E^\theta\bigg[\int_{(t,T]}\int_A \triangledown_\theta \lambda^\theta(e|s,X^{t,x,a}_{s-},I^{t,a}_{s-})  N^{t,x,a,\theta}_s \hat\nu^c(ds,de)\bigg].\label{eq dtheta J second term}
\end{align}
For the last term in \eqref{eq dtheta J step 2}, focusing on the not-almost-sure jumps of $\nu$, and using that $\nu(\{s\}\times A) = 0$ implies that $I^{t,a}_{s-} = I^{t,a}_s$ and then by assumption also $X^{t,x,a}_{s-} = X^{t,x,a}_s$, which implies that $M^{t,x,a,\theta}_s = N^{t,x,a,\theta}_s$, we obtain
\begin{align}
    &\E^\theta\bigg[\sum_{s\in (t,T], 0 < \hat\nu(\{s\}\times A) < 1} \1_{\{\nu(\{s\}\times A) = 0\}} 
    \frac{\int_A \triangledown_\theta \lambda^\theta(e|s,X^{t,x,a}_{s-},I^{t,a}_{s-}) \hat\nu(\{s\}, de)}
    {1-\int_A \lambda^\theta(e|s,X^{t,x,a}_{s-},I^{t,a}_{s-}) \hat\nu(\{s\}, de)}
    M^{t,x,a,\theta}_s\bigg]\\
    %
    %
    &= \E^\theta\bigg[\sum_{s\in (t,T], 0 < \hat\nu(\{s\}\times A) < 1} (1 - \1_{\{\nu(\{s\}\times A) > 0\}})
    \frac{\int_A \triangledown_\theta \lambda^\theta(e|s,X^{t,x,a}_{s-},I^{t,a}_{s-}) \hat\nu(\{s\}, de)}
    {1-\int_A \lambda^\theta(e|s,X^{t,x,a}_{s-},I^{t,a}_{s-}) \hat\nu(\{s\}, de)}
    N^{t,x,a,\theta}_s\bigg].\label{eq dtheta J third term first rewrite}
\end{align}
To simply this term, we note that since $\nu$ is a simple random counting measure
\begin{align}
    &\E^\theta\bigg[\sum_{s\in (t,T], 0 < \hat\nu(\{s\}\times A) < 1} \1_{\{\nu(\{s\}\times A) > 0\}}
    \frac{\int_A \triangledown_\theta \lambda^\theta(e|s,X^{t,x,a}_{s-},I^{t,a}_{s-}) \hat\nu(\{s\}, de)}
    {1-\int_A \lambda^\theta(e|s,X^{t,x,a}_{s-},I^{t,a}_{s-}) \hat\nu(\{s\}, de)}
    N^{t,x,a,\theta}_s\bigg]\\
    &= \E^\theta\bigg[\int_{(t,T]}\int_A \1_{\{0 < \hat\nu(\{s\}\times A) < 1\}}
    \frac{\int_A \triangledown_\theta \lambda^\theta(e|s,X^{t,x,a}_{s-},I^{t,a}_{s-}) \hat\nu(\{s\}, de)}
    {1-\int_A \lambda^\theta(e|s,X^{t,x,a}_{s-},I^{t,a}_{s-}) \hat\nu(\{s\}, de)}
    N^{t,x,a,\theta}_s\nu(\{s\}, du)\bigg]\\
    &= \E^\theta\bigg[\sum_{s\in (t,T], 0 < \hat\nu(\{s\}\times A) < 1}
    \frac{\int_A \triangledown_\theta \lambda^\theta(e|s,X^{t,x,a}_{s-},I^{t,a}_{s-}) \hat\nu(\{s\}, de)}
    {1-\int_A \lambda^\theta(e|s,X^{t,x,a}_{s-},I^{t,a}_{s-}) \hat\nu(\{s\}, de)}
    N^{t,x,a,\theta}_s\int_A \lambda^\theta(u|s,X^{t,x,a}_{s-},I^{t,a}_{s-})\hat\nu(\{s\}, du)\bigg],
\end{align}
using that $N^{t,x,a,\theta}$ is by construction predictable. Therefore, we can rewrite \eqref{eq dtheta J third term first rewrite} as
\begin{align}
    &\E^\theta\bigg[\sum_{s\in (t,T], 0 < \hat\nu(\{s\}\times A) < 1} \1_{\{\nu(\{s\}\times A) = 0\}}
    \frac{\int_A \triangledown_\theta \lambda^\theta(e|s,X^{t,x,a}_{s-},I^{t,a}_{s-}) \hat\nu(\{s\}, de)}
    {1-\int_A \lambda^\theta(e|s,X^{t,x,a}_{s-},I^{t,a}_{s-}) \hat\nu(\{s\}, de)}
    M^{t,x,a,\theta}_s\bigg]\\
    &= \E^\theta\bigg[\sum_{s\in (t,T], 0 < \hat\nu(\{s\}\times A) < 1} \int_A \triangledown_\theta \lambda^\theta(e|s,X^{t,x,a}_{s-},I^{t,a}_{s-}) N^{t,x,a,\theta}_s\hat\nu(\{s\}, de)\bigg].\label{eq dtheta J third term}
\end{align}
To continue, we need an auxiliary result, for which we will take a closer look at the times where $\hat\nu(\{s\}\times A) = 1$.
Since $\lambda^\theta$ is an admissible control, this implies for such time points that $\int_A \lambda^\theta(e|s,X^{t,x,a}_{s-},I^{t,a}_{s-}) \hat\nu(\{s\},de) = 1$ for all $\theta\in\Theta$ and thus
\[
\int_A \triangledown_\theta \lambda^\theta(e|s,X^{t,x,a}_{s-},I^{t,a}_{s-}) \hat\nu(\{s\}, de) = \triangledown_\theta \bigg(\int_A \lambda^\theta(e|s,X^{t,x,a}_{s-},I^{t,a}_{s-}) \hat\nu(\{s\},de)\bigg) = 0.
\]
This leads us to
\begin{align}
    \E^\theta\bigg[\sum_{s\in (t,T], \hat\nu(\{s\}\times A) = 1} \int_A N^{t,x,a,\theta}_s \triangledown_\theta \lambda^\theta(e|s,X^{t,x,a}_{s-},I^{t,a}_{s-}) \hat\nu(\{s\}, de)\bigg] &= 0.\label{eq dtheta J fourth term}
\end{align}

Thus, putting \eqref{eq dtheta J fourth term,eq dtheta J second term,eq dtheta J third term} together, we obtain that
\begin{align}
    & \E^\theta\bigg[\int_{(t,T]} \int_A \triangledown_\theta \lambda^\theta(e|s,X^{t,x,a}_{s-},I^{t,a}_{s-}) M^{t,x,a,\theta}_s \hat\nu^c(ds,de)  \\
    & \qquad \;  +  \;  \sum_{s\in (t,T], 0 < \hat\nu(\{s\}\times A) < 1} \1_{\{\nu(\{s\}\times A) = 0\}} 
    \frac{\int_A \triangledown_\theta \lambda^\theta(e|s,X^{t,x,a}_{s-},I^{t,a}_{s-}) \hat\nu(\{s\}, de)}{1-\int_A \lambda^\theta(e|s,X^{t,x,a}_{s-},I^{t,a}_{s-}) \hat\nu(\{s\}, de)} M^{t,x,a,\theta}_s
    \bigg]\\
    &=\E^\theta\bigg[\int_{(t,T]}\int_A\triangledown_\theta \lambda^\theta(e|s,X^{t,x,a}_{s-},I^{t,a}_{s-}) N^{t,x,a,\theta}_s\hat\nu(ds,de)\bigg].
\end{align}
Now using that the integrand $N^{t,x,a,\theta}$ is predictable, we can apply that $\lambda^\theta(e|s,X_{s-}^{t,x,a},I_{s-}^{t,a})\hat\nu(ds,de)$
is the predictable projection of $\nu$ under $\P^\theta$, to obtain
\begin{align}
    &\E^\theta\bigg[\int_{(t,T]}\int_A \triangledown_\theta \lambda^\theta(e|s,X^{t,x,a}_{s-},I^{t,a}_{s-}) N^{t,x,a,\theta}_s \hat\nu(ds,de)\bigg]\\
    &=\E^\theta\bigg[\int_{(t,T]}\int_A \triangledown_\theta (\log\lambda^\theta)(e|s,X^{t,x,a}_{s-},I^{t,a}_{s-}) N^{t,x,a,\theta}_s \nu(ds,de)\bigg].
\end{align}
Finally, together with \eqref{eq dtheta J step 2}, we obtain
\begin{align}
    \nabla_\theta J(t,x,a,\theta)
    &=\E^\theta\bigg[\int_{(t,T]}\int_A \triangledown_\theta (\log\lambda^\theta)(e|s,X^{t,x,a}_{s-},I^{t,a}_{s-})\cdot \big(M^{t,x,a,\theta}_s - N^{t,x,a,\theta}_s\big) \nu(ds,de)\bigg]\\
    &=\E^\theta\bigg[\int_{(t,T]}\int_A \triangledown_\theta (\log\lambda^\theta)(e|s,X^{t,x,a}_{s-},I^{t,a}_{s-})\\
    &\qquad\; \cdot \Big(J(s,X^{t,x,a}_s,e,\theta) - J(s,X^{t,x,a}_{s-},I^{t,a}_{s-},\theta) - c(s,X^{t,x,a}_{s-},I^{t,a}_{s-},e)\Big) \nu(ds,de)\bigg].
\end{align}
\end{proof}

We will see that we can use this to construct policy gradient (PG) steps for a diverse class of control problems in continuous time.


\subsection{Actor-critic algorithm}\label{subsection randomised actor critic algorithm}

We now aim to design an actor-critic (AC) learning algorithm for our randomised problem. AC algorithms are useful to tackle problems in environments where explicit knowledge of system dynamics is unavailable (e.g. model-free settings) and they consist out of two steps which are executed in turns: the policy evaluation (PE) step updates our reward functional estimate $J$ based on the current policy $\lambda$, and the policy gradient (PG) step updates our current policy $\lambda$ using the current estimate of $J$. This enables simultaneous learning of the optimal parameters $\kappa$ and $\theta$ for our parametrised families $(J^\kappa)_\kappa$ representing the reward functional and $(\lambda^\theta)_\theta$ for the optimal intensity control. In particular, we expect $J^\kappa$ to approximate the true value function for our control problem.

We will base the policy gradient (PG) step, on the representation \eqref{eq policy gradient equation} of the gradient, which we developed in \cref{subsection theory background}.  
To fit a model-free setting, we consider that in general we do not know the exact form of $c$, but instead that at any point in time $s\in [t,T]$, we are able to observe our cumulative reward up to the current time,
\[
R^{t,x,a}_s \coloneqq \int_t^s f(r,X^{t,x,a}_r,I^{t,a}_r) dr - \int_{(t,s]}\int_A c(r,X^{t,x,a}_{r-},I^{t,a}_{r-},e) \nu(dr,de).
\]
This enables us, by observing our accumulated reward right before and after we change our action, so $R^{t,x,a}_{\tau_n-}$ before the jump and $R^{t,x,a}_{\tau_n}$ after the jump, to compute the cost term appearing in \eqref{eq policy gradient equation} as follows
\[
c(\tau_n,X^{t,x,a}_{\tau_n-},I^{t,a}_{\tau_n-},I^{t,a}_{\tau_n}) = R^{t,x,a}_{\tau_n-} - R^{t,x,a}_{\tau_n},
\]
and thus obtain the following formula for the policy gradient,
\begin{align}
    \nabla_\theta J(t,x,a,\theta)
    &=\E^\theta\bigg[\int_{(t,T]}\int_A \triangledown_\theta (\log\lambda^\theta)(e|s,X^{t,x,a}_{s-},I^{t,a}_{s-})\\
    &\qquad \cdot \Big(J(s,X^{t,x,a}_s,e,\theta) - J(s,X^{t,x,a}_{s-},I^{t,a}_{s-},\theta) + R^{t,x,a}_s - R^{t,x,a}_{s-} \Big) \nu(ds,de)\bigg].
\end{align}

For the policy evaluation (PE) step, we can for example utilise a martingale loss function based on approach introduced in \autocite{jia_policy_2021-1}. Their approach is based on the observation that for the true value function $v^\theta$ $:=$ $J(\cdot,\theta)$ associated with a fixed policy $\lambda^\theta$, the process
\[
(v^\theta(s,X^{t,x,a}_s,I^{t,a}_s) + R^{t,x,a}_s)_{s\in [t,T]}
\]
is a martingale under $\P^\theta$,  where $(X^{t,x,a},I^{t,a})$ follow the intensity policy $\lambda^\theta$. Now using that at terminal time the value function $v^\theta$ is just the terminal reward $g$, so
\[
v^\theta(T,X^{t,x,a}_T,I^{t,a}_T) + R^{t,x,a}_T
= g(X^{t,x,a}_T,I^{t,a}_T) + R^{t,x,a}_T,
\]
we can conclude that for all $s\in [t,T]$, under $\P^\theta$,
\[
v^\theta(s,X^{t,x,a}_s,I^{t,a}_s) + R^{t,x,a}_s
= \argmin_{\substack{\xi\text{ is }\c F^{X,\nu}_s\text{-measurable}}} \E^\theta\big[|g(X^{t,x,a}_T,I^{t,a}_T) + R^{t,x,a}_T - \xi|^2\big].
\]
Since $J^\kappa$ is intended to approximate $v^\theta$, this motivates us to consider the following martingale loss for our learned reward functional $J^\kappa$,
\[
\text{ML}(J^\kappa) \coloneqq \frac 1 2 \E^\theta\bigg[\int_{(t,T]} \int_A \big|J^\kappa(s,X^{t,x,a}_s,I^{t,a}_s) + R^{t,x,a}_s - g(X^{t,x,a}_T,I^{t,a}_T) - R^{t,x,a}_T \big|^2 \nu(ds,de)\bigg].
\]

This loss, in essence, quantifies how much we deviate from the martingale characterisation above. Another possible choice for ML$(J^\kappa)$ would e.g. be $\E^\theta\big[\int_t^T \big|J^\kappa(s,X^{t,x,a}_s,I^{t,a}_s) + R^{t,x,a}_s - g(X^{t,x,a}_T,I^{t,a}_T) - R^{t,x,a}_T \big|^2 ds\big]$, which has been considered by \autocite{jia_policy_2021-1}.
To be able to learn the reward functional for a policy $\theta$, we will update our estimate $\kappa$ using the martingale loss $\text{ML}(J^\kappa)$ by computing
\begin{align}
&\nabla_\kappa \text{ML}(J^\kappa)\\
&= \E^\theta\bigg[\int_{(t,T]} \int_A  \Big(J^\kappa(s,X^{t,x,a}_s,I^{t,a}_s) + R^{t,x,a}_s - g(X^{t,x,a}_T,I^{t,a}_T) - R^{t,x,a}_T \Big) \nabla_\kappa  J^\kappa(s,X^{t,x,a}_s,I^{t,a}_s) \nu(ds,de)\bigg].
\end{align}

By combining both steps, we then obtain the following generic actor-critic algorithm for randomised control problem in a model-free setting. 

\vspace{3mm}

\begin{algorithm}[H]
    \caption{Offline-episodic actor-critic algorithm}
    \label{algorithm actor critic algorithm for randomised problems}
    \SetAlgoLined
    \DontPrintSemicolon
    \KwIn{initial state $x_0$, initial action $a_0$, parametrised family of reward functions $(J^\kappa)_\kappa$, parametrised family of randomised intensity actions $(\lambda^\theta)_\theta$, initial learning rates $\eta_\kappa,\eta_\theta$, learning schedule $l(\cdot)$}
    \KwOut{learned value function $J^\kappa$, optimal randomised control $\lambda^\theta$}

    initialise $\kappa,\theta$\;
    
    \For{episode $j=1,\dotso$}
    {
    simulate $(X_t,I_t)_{t\in [0,T]}$ starting from $(X_0, I_0) = (x_0, a_0)$ according to the policy $\lambda^\theta$ and observe the accumulated running reward $(R_t)_{t\in [0,T]}$ and the terminal reward $G_T$ $=$ $g(X_T,I_T)$\;
    compute $\nabla_\kappa \text{ML}(J^\kappa) \gets \sum_{I_{t-}\not= I_t} (J^\kappa(t,X_t,I_t) + R_t - G_T - R_T) \nabla_\kappa J^\kappa(t,X_t,I_t)$\;
    compute $\triangledown_\theta J^\kappa\gets \sum_{I_{t-}\not= I_t} (J^\kappa(t,X_t,I_t) - J^\kappa(t,X_{t-},I_{t-}) + R_t - R_{t-}) \triangledown_\theta (\log \lambda^\theta)(I_t| t,X_{t-},I_{t-})$ \;
    update $\kappa \gets \kappa - \eta_\kappa l(j) \nabla_\kappa \text{ML}(J^\kappa) $\;
    update $\theta \gets \theta + \eta_\theta l(j) \triangledown_\theta J^\kappa$\;
    }
\end{algorithm}

\vspace{3mm}
Finally, it is important to note that while this algorithm addresses the randomised problem, our primary interest lies in (non-randomised) stochastic control problems; their randomised counterparts serve as tools for handling these control problems. Specifically, our objective will in general not be to find the optimal intensity $\lambda^\theta$, but rather to find the optimal (non-randomised) control $\alpha$. Therefore, in the next \cref{section application to stochastic control problems}, we will discuss in more detail how to utilise this algorithm to find the optimal control for the corresponding stochastic control problems.


\section{Application to stochastic control problems}\label{section application to stochastic control problems}

In this section, we consider general stochastic control problems for which  a randomised formulation in the form of \cref{section unified approach} exists.  The classical problem is the case of controlled Markov processes  $X^\alpha$, e.g., driven by diffusion processes, with regular controls $\alpha$ valued in $A$, and where the objective is to maximise over $\alpha$ a criterion in the form 
\begin{align}
\mrJ(\alpha) &= \; \E \Big[ g(X_T^\alpha) + \int_0^T f(t,X_t^\alpha,\alpha_t) dt\Big].     
\end{align}
The corresponding randomised formulation is the one described in \cref{section unified approach} with $g(x)$ depending only on $x$, $c$ $\equiv$ $0$, and the key result, proved in \autocite{kharroubi_feynmankac_2015}, see  also \cite{fuhpha15}, is the statement that the two value functions coincide, namely: 
\begin{align}
\sup_\alpha \mrJ(\alpha) &= \; \sup_{\lambda \in {\cal V}} \E^\lambda \Big[ g(X_T) + \int_0^T f(t,X_t,I_t) dt\Big].    
\end{align}
Such randomised formulations have been developed for a large class of continuous time control problems, including, but not limited to,  impulse control problems in \autocite{khamaphazha10},  
optimal stopping in \autocite{fuhrman_representation_2016}, or optimal switching problems in \autocite{bouchard_stochastic_2009,elie_probabilistic_2010} as it  will be illustrated in the next section. 
The core idea behind the randomisation framework is to replace the control by a random point process, usually a Poisson point process, whose intensity becomes the new control, which results in a randomised problem in form of \cref{section unified approach}. The advantage of this randomised formulation is that it provides a unified framework for many different classes of control problems, and not only for continuous time problems but even includes stochastic control in discrete time on deterministic and/or random grids.
Our goal is then to develop an actor-critic algorithm for the original stochastic control problem by utilising its randomised counterpart and its actor-critic algorithm derived in \cref{subsection randomised actor critic algorithm}. However, the drawback is that the randomised problem is not directly equivalent to the original problem. While one can often show that the value functions of both problems coincide, it is not trivial how to recover an optimal control for the original problem just from studying the randomised formulation. In particular, it is also not clear whether the randomised formulation even has an optimal (randomised) control -- in general this will not be the case.

Let us for simplicity of presentation assume that $\hat\nu(ds,de) = \bar\nu(ds)\mu(de|s,X^{}_{s-},I^{}_{s-})$, where $\bar\nu$ and $\mu$ are both non-random. Then $\bar\nu$ describes the distribution of points in $\nu$ in time, and $\mu$ is the kernel transition probability describing the mark distributions of such points. Similarly, we split $\lambda^\theta$ into an intensity $\Lambda^\theta$ for new points and a probability density $\bar\lambda^\theta$ for their marks as follows
\[
\Lambda^\theta(s,x,a) \coloneqq \int_A \lambda^\theta(e|s,x,a) \mu(de | s,x,a),
\;\;  \bar\lambda^\theta(e|s,x,a) := \frac{\lambda^\theta(e|s,x,a)}{\Lambda^{\theta}(s,x,a)},\quad (s,x,a) \in [0,T]\times\R^d\times A.
\]

Let us further assume that $\mu(\{a\}|s,x,a) > 0$ for all $(s,x,a)\in [0,T]\times\R^d\times A$, so that at any point there is a non-negative probability that the jump of $\bar\nu$ does not induce a real jump in $I$.
Given a function $\Lambda$, we now denote by $\Theta_{\leq\Lambda}$ (resp.\@ $\Theta_{=\Lambda}$) the set of all $\theta\in\Theta$ such that $\Lambda^\theta \leq \Lambda$ (resp.\@ $\Lambda^\theta = \Lambda$). Then we note that for every $\theta\in\Theta_{\leq\Lambda}$, the process $\lambda(e|s,x,a) \coloneqq \lambda^\theta(e|s,x,a) + (\Lambda(s,x,a)-\Lambda^\theta(s,x,a)) \frac 1 {\mu(\{a\}|s,x,a)} \in\c V$ emulates the control $\lambda^\theta$ in the sense that $\P^\lambda_{(X,I)} = \P^{\theta}_{(X,I)}$. Therefore, supposing that $(\lambda_\theta)_{\theta\in\Theta}$ is sufficiently dense in $\c V$, we see that $\lambda$ can be approximated by $(\lambda^{\theta_n})_n$ such that $\theta_n\in \Theta_{=\Lambda}$, which shows that
\begin{align}\label{eq J over Theta leq eq Lambda}
\sup_{\theta\in\Theta_{\leq\Lambda}} J(t,x,a,\theta) = \sup_{\theta\in\Theta_{=\Lambda}} J(t,x,a,\theta).
\end{align}
This motivates us to view $\Lambda$ as some kind of inverse step size, since as $\Lambda\to\infty$ (resp. $\Lambda\bar\nu(\{s\})\uparrow 1$ if $\bar\nu(\{s\}) > 0$), we see that
\[
\sup_{\theta\in\Theta_{=\Lambda}} J(t,x,a,\theta) \to 
\sup_{\theta\in\Theta} J(t,x,a,\theta).
\]
Thus, we introduce the following restricted parameter sets for $\theta$, for $\Lambda_c\geq 0$ and $0\leq \Lambda_d\leq 1$,
\[
\Theta_{\Lambda_c,\Lambda_d}
= \big\{\theta \in\Theta \;\big|\; \Lambda^\theta(s,x,a)
= \Lambda_c \1_{\{\bar\nu(\{s\}) = 0\}} + \tfrac{\Lambda_d}{\bar\nu(\{s\})} \1_{\{\bar\nu(\{s\}) > 0\}}, \text{ for all }(s,x,a)\big\}
\subseteq \Theta.
\]

Then, $\sup_{\theta\in\Theta_{\Lambda_c,\Lambda_d}} J \to \sup_{\theta\in\Theta} J$ as long as $\Lambda_c\to\infty$ and $\Lambda_d \uparrow 1$.
Further, since while optimising over $\Theta_{\Lambda_c,\Lambda_d}$, the intensity is fixed, it is equivalent to optimise instead over the probability densities $(\bar\lambda_\theta)_{\theta\in\Theta_{\Lambda_c,\Lambda_d}}$, which by construction satisfy $\int_A \bar\lambda^\theta(e|\cdot) \mu(de|\cdot) \equiv 1$ and $\bar\lambda^\theta\geq 0$. Note that this family, if $\Theta$ is sufficiently exhaustive, does not actually depend on $\Lambda_c,\Lambda_d$ anymore.
Thus, by imposing an intensity schedule $\Lambda_c$ and $\Lambda_d$ ensuring that $\Lambda_c\to\infty$ and $\Lambda_d\uparrow 1$, we obtain \cref{algorithm actor critic algorithm with random grids and intensity schedule}.

At the same time, we recall that solving the randomised problem was however not our original goal. Instead, it serves as a tool for solving the original (non-randomised) problem. In particular, the $\bar\lambda^\theta$ we obtain from \cref{algorithm actor critic algorithm for randomised problems} is not our desired control.
We recall that instead $\bar\lambda^\theta$ represents the intensity for $I^{}$, and the process $I^{}$ then actually plays the role of the sought-after control process $\alpha$. Therefore, let us define for each $\theta$ also a control $\alpha^\theta$ as e.g.\@ the $\argmax$ of the distribution $\bar\lambda^\theta(e | s,x,a) \mu(de | s,x,a)$. The motivation is that at each jump $s$ of $I^{}$, we draw our new control $I^{}_s$ from the distribution $\bar\lambda^\theta(e | s, X^{}_{s-},I^{}_{s-}) \mu(de | s,X^{}_{s-},I^{}_{s-})$, and if the intensity $\Lambda^\theta\to\infty$ (resp. $\Lambda^\theta\bar\nu(\{s\}) \uparrow 1$ if $\bar\nu(\{s\}) > 0$), then we are essentially able to draw a new control at every time point $s\in [t,T]$, just as in the original control problem. Consequently, letting in our case $\Lambda_c\to\infty$ and $\Lambda_d\uparrow 1$,
this then leads to the convergence of $\alpha^\theta\to\alpha_*$.

\vspace{3mm}

\begin{algorithm}[H]
    \caption{\footnotesize{Offline-episodic actor-critic algorithm with random grids and intensity schedule}}
    \label{algorithm actor critic algorithm with random grids and intensity schedule}
    \SetAlgoLined
    \DontPrintSemicolon
    \KwIn{initial state $x_0$, initial action $a_0$, terminal time $T$, parametrised family of reward functions $(J^\kappa)_\kappa$, baseline random grid sampling distribution $\bar\nu$, baseline action distribution kernel $\mu$, parametrised family of action densities $(\bar\lambda^\theta)_\theta$, intensity schedule $\Lambda_c(\cdot)$, $\Lambda_d(\cdot)$, initial learning rates $\eta_\kappa,\eta_\theta$, learning schedule $l(\cdot)$}
    \KwOut{learned value function $J^\kappa$, optimal randomised control $\lambda^\theta$, optimal (non-randomised) control $\alpha^\theta$}

    initialise $\kappa,\theta$\;
    
    \For{episode $j=1,\dotso$}
    {
    initialise $\tau_0\gets 0$, $r_0\gets 0$\;
    simulate point process $U$ on $(0,T]$ with stochastic intensity $\Lambda_c(j)\1_{\{\bar\nu(\{s\})=0\}} \bar\nu(ds) + \Lambda_d(j) \1_{\{\bar\nu(\{s\}) > 0\}} \delta_s(ds)$
    $\to$ obtain grid points $(\tau_n)_{n=1,\dotso,N}$\; 
    \For{$n=1,...,N$}{
        simulate $X_{[\tau_{n-1},\tau_n)}$ from $X_{\tau_{n-1}} = x_{n-1}$ with control $a_{n-1}$\;
        observe the new state $x_{n-} \gets X_{\tau_n-}$ and accumulated running reward $r_{n-} \gets R_{\tau_n-}$ at time $\tau_n-$\;
        simulate and update the new control $a_n \sim \bar\lambda^\theta(e|\tau_n,x_{n-},a_{n-1})\mu(de | \tau_n,x_{n-},a_{n-1})$\;
        observe the new state $x_n \gets X_{\tau_n}$ and accumulated running reward $r_n \gets R_{\tau_n}$ after updating the control at time $\tau_n$\;
    }
    simulate $X_{[\tau_N,T]}$ with control $a_N$\;
    observe the final state $x_{N+1}\gets X_T$ and set $a_{N+1}\gets a_N$, $\tau_{N+1}\gets T$\;
    observe the final accumulated running reward $r_{N+1}\gets R_T$ and the terminal reward $G_T$ $=$ $g(X_T,I_T)$ at time $T$\;
    compute $\nabla_\kappa \text{ML}(J^\kappa) \gets\sum_{a_n\not= a_{n-1}} (r_n + J^\kappa(\tau_n,x_n,a_n) - G_T - r_{N+1}) \triangledown_\kappa J^\kappa(\tau_n,x_n,a_n)$\;
    compute $\nabla_\theta J^\kappa\gets \sum_{a_n\not= a_{n-1}} (J^\kappa(\tau_n,x_n,a_n) - J^\kappa(\tau_n,x_{n-},a_{n-1}) + r_n - r_{n-}) \triangledown_\theta (\log \bar\lambda^\theta)(a_n| \tau_n,x_{n-},a_{n-1})$ \;
    update $\kappa \gets \kappa - \eta_\kappa l(j) \nabla_\kappa \text{ML}(J^\kappa)  $\;
    update $\theta \gets \theta + \eta_\theta l(j) \nabla_\theta J^\kappa$\;
    }
    obtain $\alpha^\theta(t,x,a)$ as the $\argmax$ of the probability distribution $\bar\lambda^\theta(e|t,x,a)\mu(de|t,x,a)$\;
\end{algorithm}

\vspace{3mm}

Finally, we want to conclude this section by giving a version of the actor-critic algorithm but with a \grqq fixed step size\grqq{}. This is the version we will also use later in \cref{section numerical experiments}. So we will fix the intensity $\Lambda$ and thus the base process $\tilde\nu\coloneqq \Lambda\bar\nu$, and now optimise again over the parametrised family of action densities $(\bar\lambda^\theta)_\theta$, which results in the following \cref{algorithm actor critic algorithm with random grids}.

\begin{remark}
    This leads to a flexible framework accommodating various types of sampling grids, such as
    \begin{itemize}
        \item deterministic discrete grids, by choosing $\tilde\nu = \sum_{k=0}^N \delta_{\frac {kT} N}(ds)$, for $N\geq 2$,
        \item random discrete grids, by setting $\tilde\nu = \sum_{k=0}^N p_{samp} \delta_{\frac {kT} N}(ds)$, for $N\geq 2$ and $p_{samp} \in (0,1]$,
        \item Poisson grids are possible with $\tilde\nu = \lambda ds$ for some intensity $\lambda > 0$.
    \end{itemize}
    In \cref{section numerical experiments}, we will compare the choice of deterministic and random discrete grids through two numerical examples of optimal switching problems.
\end{remark}

\begin{algorithm}[H]
    \caption{Offline-episodic actor-critic algorithm with random grids}
    \label{algorithm actor critic algorithm with random grids}
    \SetAlgoLined
    \DontPrintSemicolon
    \KwIn{initial state $x_0$, initial action $a_0$, terminal time $T$, parametrised family of reward functions $(J^\kappa)_\kappa$, random grid sampling distribution $\tilde\nu$, baseline action distribution kernel $\mu$, parametrised family of action densities $(\bar\lambda^\theta)_\theta$, initial learning rates $\eta_\kappa,\eta_\theta$, learning schedule $l(\cdot)$}
    \KwOut{learned value function $J^\kappa$, optimal randomised control $\lambda^\theta$, optimal (non-randomised) control $\alpha^\theta$}

    initialise $\kappa,\theta$\;
    
    \For{episode $j=1,\dotso$}
    {
    initialise $\tau_0\gets 0$, $r_0\gets 0$\;
    simulate point process on $(0,T]$ with stochastic intensity $\tilde\nu$ $\to$ obtain grid points $(\tau_n)_{n=1,\dotso,N}$\;
    \For{$n=1,...,N$}{
        simulate $X_{[\tau_{n-1},\tau_n)}$ from $X_{\tau_{n-1}} = x_{n-1}$ with control $a_{n-1}$\;
        observe the new state $x_{n-} \gets X_{\tau_n-}$ and accumulated running reward $r_{n-} \gets R_{\tau_n-}$ at time $\tau_n-$\;
        simulate and update the new control $a_n \sim \bar\lambda^\theta(e|\tau_n,x_{n-},a_{n-1})\mu(de | \tau_n,x_{n-},a_{n-1})$\;
        observe the new state $x_n \gets X_{\tau_n}$ and accumulated running reward $r_n \gets R_{\tau_n}$ after updating the control at time $\tau_n$\;
    }
    simulate $X_{[\tau_N,T]}$ with control $a_N$\;
    observe the final state $x_{N+1}\gets X_T$ and set $a_{N+1}\gets a_N$, $\tau_{N+1}\gets T$\;
    observe the final accumulated running reward $r_{N+1} \gets R_T$ and the terminal reward $G_T$ $=$ $g(X_T,I_T)$ at time $T$\;
    \;
    compute $\nabla_\kappa \text{ML}(J^\kappa) \gets\sum_{a_n\not= a_{n-1}} (r_n + J^\kappa(\tau_n,x_n,a_n) - G_T - r_{N+1}) \triangledown_\kappa J^\kappa(\tau_n,x_n,a_n)$\;
    compute $\nabla_\theta J^\kappa\gets \sum_{a_n\not= a_{n-1}} (J^\kappa(\tau_n,x_n,a_n) - J^\kappa(\tau_n,x_{n-},a_{n-1}) + r_n - r_{n-}) \triangledown_\theta (\log \bar\lambda^\theta)(a_n| \tau_n,x_{n-},a_{n-1})$ \;
    update $\kappa \gets \kappa - \eta_\kappa l(j) \nabla_\kappa \text{ML}(J^\kappa)$\;
    update $\theta \gets \theta + \eta_\theta l(j) \nabla_\theta J^\kappa$\;
    }

    define $\alpha^\theta(t,x,a)$ as the $\argmax$ of the probability distribution $\bar\lambda^\theta(e|t,x,a)\mu(de|t,x,a)$\;
\end{algorithm}

\section{Numerical experiments for switching problems using neural networks}\label{section numerical experiments}

\subsection{Optimal switching problem}\label{subsection optimal switching}

In this paragraph, we recall the  connection between  optimal switching problems with their randomised formulation following \autocite{bouchard_stochastic_2009}. Note that this randomisation method has been extended to large class of further problems including impulse control, optimal stopping and regular control problems, and thus such problems also fit into the unified setting for our policy gradient method.

Let $(\Omega,\c F,\P)$ be a complete probability space carrying an $m$-dimensional Brownian motion $W$ and let $\b F^W$ be its generated filtration. Let $A = \{1,\dots,N\}$ for some fixed $N\in\N$ denote the finite action set. A switching control is an $A$-valued piece-wise constant process of the form
\[
\alpha = a \1_{[t,\tau_0)} + \sum_{n\in\N} \xi_n \1_{[\tau_n,\tau_{n+1})},
\]
where $(\tau_n)_n$ is an increasing sequence of stopping times such that $\tau_n\to\infty$ $\P$-a.s., $(\xi_n)_n$ is a sequence of $A$-valued random variables such that $\xi_n$ is $\c F^W_{\tau_n}$-measurable and $a\in A$ is a fixed initial control.
Given an initial value $x\in\R^d$, we further consider the controlled state dynamics as the solution to the stochastic differential equation
\begin{align}
\label{eq optimal switching state dynamics}
    X^{t,x,\alpha}_s = x + \int_t^s b(r,X^{t,x,\alpha}_r,\alpha_r) dr + \int_t^s \sigma(r,X^{t,x,\alpha}_r,\alpha_r) dW_r + \sum_{t < \tau_n \leq s} \gamma(\tau_n,X^{t,x,\alpha}_{\tau_n-},\alpha_{\tau_n-},\alpha_{\tau_n}).
\end{align}
We denote by $\c A$ as the set of as all switching controls $\alpha$.
Note that together with the following \cref{assumptions control problem state dynamics}, this ensures that the above state dynamics \eqref{eq optimal switching state dynamics} are well-defined.

\begin{assumption}\label{assumptions control problem state dynamics}
The coefficients $b$, $\sigma$ and $\gamma$ are Lipschitz and of linear growth w.r.t. $x$: there exists a positive constant $C$ s.t. for all $t$ $\in$ $[0,T]$, $x,x'$ $\in$ $\R^d$, $a,a'$ $\in$ $A$, 
\begin{align}
|b(t,x,a) - b(t,x',a)| + |\sigma(t,x,a) - \sigma(t,x',a)| + 
|\gamma(t,x,a,a') - \gamma(t,x',a,a')| & \leq C |x-x'|,\\ 
|b(t,x,a)| + |\sigma(t,x,a)| + |\gamma(t,x,a,a')| & \leq C(1 + |x|). 
\end{align}
\end{assumption}

Our goal is now to maximise the following reward functional
\[
\mrJ(t,x,a,\alpha) \coloneqq \E\bigg[ g(X^{t,x,\alpha}_T,\alpha_T) + \int_t^T f(s,X^{t,x,\alpha}_s,\alpha_s) ds - \sum_{t < \tau_n \leq T} c(\tau_n,X^{t,x,\alpha}_{\tau_n-},\alpha_{\tau_n-},\alpha_{\tau_n}) \bigg],
\]
and we define the value function as follows
\[
V(t,x,a) \coloneqq \sup_{\alpha\in\c A} \mrJ(t,x,a,\alpha).
\]

We make the standard assumptions on the gain and cost functions: 

\begin{assumption}\label{assumptions control problem reward functional}
The reward functions $f$, $g$ and the cost function are continuous w.r.t. the $x$ argument with quadratic growth condition: there exists some positive constant $C$ s.t. for all $t$ $\in$ $[0,T]$, $x$ $\in$ $\R^d$, $a,a'$ $\in$ $A$, 
\begin{align}
|f(t,x,a)| + |g(x,a)| + |c(t,x,a,a')| & \leq C(1 + |x|^2). 
\end{align}

\end{assumption}


To formulate the randomised version of this problem as in \cref{section unified approach}, we introduce an independent random point process $\nu$ on $[0,T]\times A$ with predictable projection $\hat\nu$.
Correspondingly, we define the $A$-valued process
\[
I^{t,a}_s = a + \int_{(t,s]}\int_A (e - I^{t,a}_{r-}) \nu(dr,de),\quad s\in [t,T],
\]
which will replace our control process. In particular, the state process will follow the following uncontrolled state dynamics
\begin{align}
    X^{t,x,a}_s = x + \int_t^s b(r,X^{t,x,a}_r,I^{t,a}_r) dr + \int_t^s \sigma(r,X^{t,x,a}_r,I^{t,a}_r) dW_r + \int_{(t,s]}\int_A \gamma(r,X^{t,x,a}_{r-},I^{t,a}_{r-},e) \nu(dr,de).
\end{align}
Our set of control will instead now be the set $\c V$ of $\b F^{W,\nu}\otimes\c B(A)$-predictable, essentially bounded  processes $\lambda$ such that there exists a with respect to $\P$ absolutely continuous probability measure $\P^\lambda\ll \P$ under which $\nu$ is a random point process with predictable projection $\lambda_s(e)\hat\nu(ds,de)$, see also \eqref{eq girsanov formula} for a characterisation of such $\lambda\in\c V$.
Then the reward functional is defined by
\[
J(t,x,a,\lambda) \coloneqq \E^{\P^\lambda} \bigg[ g(X^{t,x,a}_T,I^{t,a}_T) + \int_t^T f(s,X^{t,x,a}_s,I^{t,a}_s) ds - \int_{(t,T]} c(s,X^{t,x,a}_{s-},I^{t,a}_{s-},e) \nu(ds,de) \bigg],
\]
and we introduce the following randomised value function
\[
V^{\c R}(t,x,a) \coloneqq \sup_{\lambda\in\c V} J(t,x,a,\lambda).
\]

\textcite{bouchard_stochastic_2009} studied the case where $\nu$ is a Poisson point process with compensator $\hat\nu(ds,de) = \sum_{a\in A} \delta_a(de) ds$, for which the set of admissible randomised controls $\c V$ then reduces to all $\b F^{W,\nu}\otimes\c B(A)$-predictable, essentially bounded processes $\lambda$. Under some additional regularity and growth assumptions, it is proved the following equivalence result between the optimal switching and the randomised problem.

\begin{theorem}[{\autocite[Theorem 2.1]{bouchard_stochastic_2009}}]
    Let \cref{assumptions control problem reward functional,assumptions control problem state dynamics} hold and $\nu$ have a predictable projection of the form $\hat\nu(ds,de) = \sum_{a\in A} \delta_a(de) ds$.
    Further assume that the regularity assumptions \autocite[Assumptions H1, H2]{bouchard_stochastic_2009} are satisfied.
    Then value functions of both problems coincide, that is $V = V^{\c R}$.
\end{theorem}

\begin{remark}
While usually $\nu$ is chosen as Poisson point process, any \emph{sufficiently dense} point measure, under suitable assumptions, would work for such an equivalence result.
In particular, starting from any given point process, even with deterministic atoms as in the case of a deterministic grid, and by e.g. adding additional points sampled from a Poisson point process, one would obtain such a \emph{sufficiently dense} point measure.
\end{remark}

\vspace{3mm}

In the sequel for our numerical experiments, we consider an optimal switching problem where part of the state which is not controlled is continuous, while the other part is controlled and  takes discrete values. 
Therefore the randomised controlled state is modeled by a discrete Markov chain described by  probability transitions which are functions of the global  state. At convergence, we expect that these probabilities converge either to 1 or 0,  are  discontinuous in time for a fixed state, so that   they cannot be represented as functions of time using neural networks. 
Consequently,  we use a deterministic uniform grid  of $N$ dates  on $[0,T]$. We note $t_n = n  \Delta t$  where  $\Delta t = \frac{T}{\bar N}$ with $\bar N = N-1$ the number of time steps. At each time step  and each possible state, a neural network is used to describe the transition probabilities from one state to the other on.

We then propose :
    \begin{itemize}
        \item either to sample time randomly on the deterministic grid :  the  number of  points  $\tilde N$ chosen on the grid is sampled using a binomial distribution with a probability $p_{samp}$ and a number of trials equal to $N-2$. Then the points from the grid are chosen randomly with an uniform law,
        \item or to take the $N$ points grid corresponding to  $p_{samp}=1$.
    \end{itemize}
    We denote by  $(\tau_k)_{k \ge 0}$ the random lattice sampled from the deterministic lattice  with  $\tau_0=0$,  and complete it with the convention that  $\tau_p =T$ for $k <p  \le  N$ if $\tau_k= T$. We set  $[k]$ $=$ $\frac{\tau_k}{\Delta t}$ the random grid index associated with values in $[0,N]$.

\begin{remark}
  The fact that the control has to be  modeled at each time step by different networks to get good results  for degenerated controlled states with constraints was already shown  in the case of reservoir optimization in \cite{warin2023reservoir}.  
\end{remark}

Next, depending on the problem, it may be interesting to take a representation of the reward functional $J$  different form the one proposed by $\text{ML}_\tau(J^\kappa)$ and in the different examples below we detail  the representation taken.

\vspace{1mm}

In the two examples below, we  model the energy curve
using the classical HJM model as \cite{warin2023reservoir}:
\begin{align}
    \frac{dF(t,T)}{F(t,T)} = e^{-\beta(T-t)} \sigma d W_t
    \label{eq:price}
\end{align}
where $W_t$ is a one-dimensional Brownian motion defined on a probability space $(\Omega, \mathcal{F}, \P)$. The spot price is then equal to $S_t= F(t,t)$. As numerical example, we take $T=30$ , $F(0,t)= 90+ 10 \cos( 2\pi \lfloor \frac{t}{30} \rfloor)$, 
$\beta= 0.15$, $\sigma = 0.5$.

\subsection{Starting and stopping in physical assets}
We consider the problem of a thermal asset generating power as in \cite{hamjea07},\cite{porchet2009valuation}.
The asset has a production cost of $K$  per time unit,  and has two  states : either on (state 1) or off (state 0) and 
the switching control $\alpha$ $=$ $(\alpha_t)_t$ is  
\begin{align}
\alpha_{t} & = \;  \alpha_0 \1_{[\tau_0, \tau_1)} + \sum_{n>0} \xi_n \1_{[ \tau_n,\tau_{n+1})}(t), \quad  0 <  t \leq T, 
\end{align}
 with $\alpha_0=1$ (the asset is on at $t=0$), and the random variables $(\xi_n)_n$ denote the sequence of operating regimes valued in $A$ $=$ $\{0,1\}$, representing the decisions to stop or run the production. There is a fixed cost for switching from a mode to another one, namely $c_{0,1}$  (resp. $c_{1,0}$)  for starting (resp. stopping) the production. 

The manager of the power asset aims to maximize over $\alpha$ the expected global profit: 
\begin{align}
\mrJ(\alpha) &= \; \E \Big[ \int_0^T f(S_t,\alpha_t) dt - \sum_{n} c_{\alpha_{\tau_n},\alpha_{\tau_{n+1}}} \Big], 
\end{align} 
with running profit functions $f(s,0)$ $=$ $0$, and $f(s,1)$ $=$ $ s- K$.

\begin{itemize}
    \item 
As for the control, we model the switching probability at each time step $n$ by two neural networks depending on the price, and   using a sigmoid function at the output layer, 
$\bar\lambda^{\theta_{n,i}}(S)$ takes values in $[0,1]$  for $i=0,1$ with parameters $\theta_{n,i}$ and such that for 
$i \in \{0,1\}$, $\bar\lambda^{\theta_{n,i}}$ is the probability that, given the state $i$ in $t_{n}^{-}$, the asset will change to state  $1-i$ in $t_n$. Here,  
$\theta=((\theta_{n,i})_{n=1, \bar N-1})_{i=0,1}$. 
\item As for the value function $J$   we use similarly for each time step $n$   two neural networks depending on  the price $S$:  $(J^{\kappa_{n,i}}(S))_{i=0,1}$ taking values in $\R$  with parameters $\kappa_{n,i}$ where $J^{\kappa_{n,i}}$ is the value function in state $i$ for $i \in \{0,1\}$. Here, $\kappa = ((\kappa_{n,i})_{n=0, \bar N-1})_{i=0,1}$. We also take the convention $J^{\kappa_{[N],i}}=0$ for $i=0,1$.
\end{itemize}
 For this example, to estimate the reward function $J$, we propose to minimize the loss function 
 \begin{align}
   \sum_{k=0}^{  \bar N -1}  \E \Big[\Big|  J^{\kappa_{[k],\xi_k}}( S_{\tau_k}) - \sum_{n=0}^{\bar N-1} (S_{t_n}^{\tau_k}-K) \Delta t  \1_{\{ \alpha_{t_n}^{\xi_{k}}=1\}} \1_{\{t_n \ge \tau_k\}}  +  \sum_{n \ge k} c_{\alpha_{\tau_n}^{\xi_{k}}, \alpha_{ \tau_{n+1}}^{\xi_{k}}}   \1_{\{\tau_{n+1} <T\}} \Big|^2   \Big]
     \label{eq:MLEq}
 \end{align}
 with the convention $c_{i,i}=0$, $i=0,1$ and where for $k=0, \ldots, \bar N-1$, $\xi_k$ is the state regime with values in $\{0,1\}$ which is sampled uniformly. $\alpha^{\xi_{k}}_t$ for $t \ge \tau_k$ denote the regimes  sampled  randomly using probabilities $(\bar\lambda^{\theta_{[l],\alpha^{\xi_{k}}_{\tau_{l-1}}}}(S_{\tau_{l}}))_{l > k}$  starting  at date $\tau_k$ with a value $\xi_k$, while $S_{t_n}^{\tau_k}$ is the asset value in $t_n$ obtained by sampling from its initial distribution in $\tau_k$ according to the asset law.

 The gradient function $\nabla_\theta J^\kappa$ is estimated locally for each time step and the sum of the local gradients $DW( \theta)$ is used
\begin{align}
     DW( \theta) =  \sum_{n=0}^{\bar N-2}  \E[   \nabla_\theta\log(\bar\lambda^{\theta_{n+1,\xi_n}}(S_{t_{n+1}}))  
(J^{\kappa_{n+1,\alpha_{t_{n+1}}^{\xi_n}}}(S_{t_{n+1}}) - J^{\kappa_{n+1,\xi_n}}(S_{t_{n+1}})
   -  c_{\xi_n,\alpha_{t_{n+1}}^{\xi_n}})]
 \end{align}
 where once again for each $n$ in the  loop $S_{t_{n+1}}$ is sampled according the asset law at date $t_{n+1}$,  $\xi_k$  is sampled uniformly in $\{0,1\}$, and  $\alpha_{t_{n+1}}^{\xi_n}$ is sampled according the probability $\bar\lambda^{\theta_{n+1,\xi_n}}(S_{t_{n+1}})$. 

\begin{remark}
Instead of simulating the process $X$ in forward direction as in \cref{algorithm actor critic algorithm with random grids} and evaluating $\nabla_\theta J^\kappa$, we use a local version of the gradient which randomly samples the state at every possible time step. This extra randomisation of the state allows us to be sure that all states are explored and gives better results.  This kind of extra randomisation is generally used in classical actor critic methods, where the control is taken as a normal law with decreasing variance with iterations (see for example \cite{pham2023actor}).
\end{remark}

We use the ADAM algorithm with $\eta_\theta =0.00015$ , $\eta_\kappa=0.03$. As noticed in \cite{pham2023actor}, it is crucial to have $\eta_\theta << \eta_\kappa$ to have good convergence. 
The $\tanh$ activation function is taken for the activation functions in the hidden  layers, while the sigmoid activation function is used   for the output layer for the probabilities. The batch size is equal to $10000$.
The references are calculated using dynamic programming in the StOpt library \cite{gevret2018stochastic} where regression are calculated using adapted linear regression per mesh \cite{bouchard2012monte} and $30$ time steps so $N=31$ are used. In the deterministic case, we get a value of $86$, while using $\sigma=0.15$, the value function is equal to $146.9$.
On \cref{fig:thermal}, we plot the convergence of the actor critic  algorithm in the stochastic case using $p_{samp}=1$. The convergence to the correct value is achieved after more than $10000$ gradient iterations.
In the graph, the  ``Function value" is obtained using the $J^\kappa$ approximation of $J$, while the  ``Gain expectation" is obtained using the gain estimate in the simulation (the controls and grids are sampled).

\begin{figure}[H]
    \centering
        \begin{minipage}[c]{.49\linewidth}
\includegraphics[width=\linewidth]{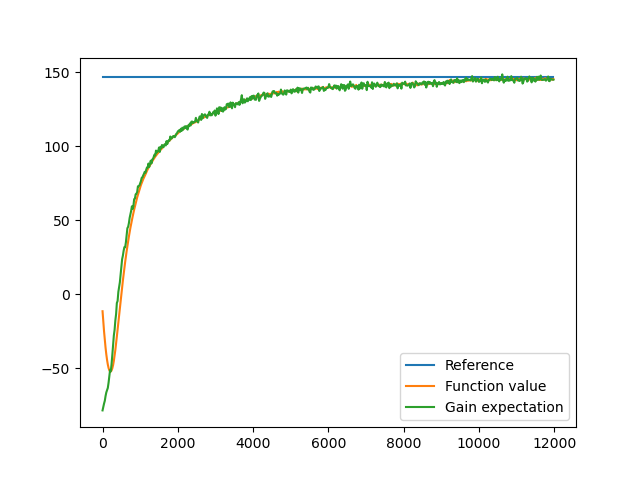}
    \caption*{$p_{samp}=1.$, $N=31$}
    \end{minipage}
    \caption{Convergence of the actor critic algorithm  for the thermal switching problem depending on the gradient iteration.}
    \label{fig:thermal}
\end{figure}

On \cref{fig:thermalsamp}, we explore the effect of extra temporal randomization. 
Not surprisingly, using a fixed lattice to sample from degrades the results as the sampling ratio $p_{samp}$ is lower. Similarly, by fixing the sampling ratio, the results improve as we increase the number of time steps of the lattice. 

\begin{figure}[H]
    \centering
        \begin{minipage}[c]{.32\linewidth}
\includegraphics[width=\linewidth]{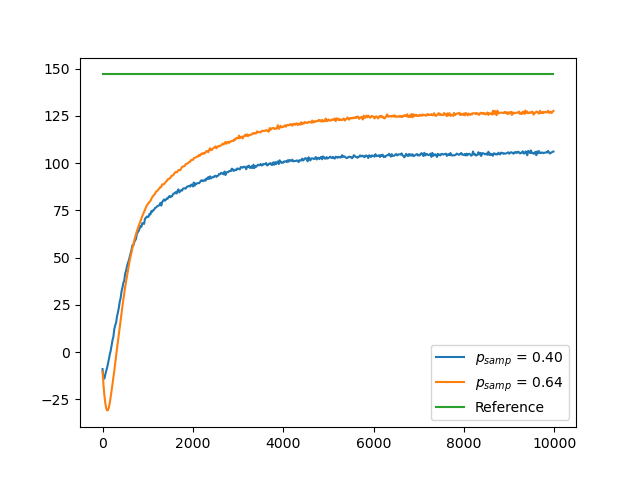}
    \caption*{$N=31$}
    \end{minipage}
           \begin{minipage}[c]{.32\linewidth}
\includegraphics[width=\linewidth]{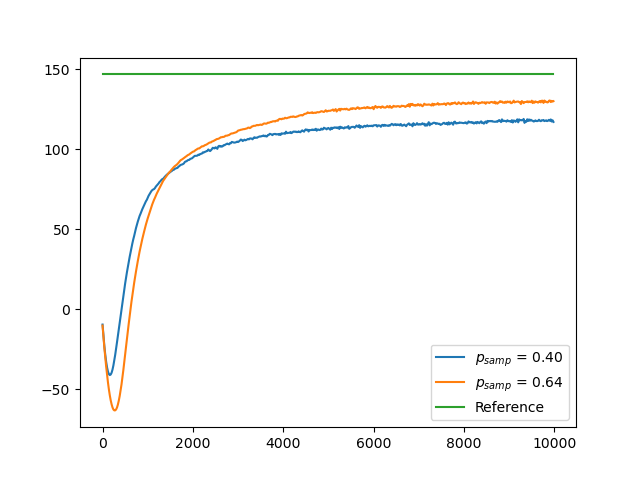}
    \caption*{$N=61$}
    \end{minipage}
            \begin{minipage}[c]{.32\linewidth}
\includegraphics[width=\linewidth]{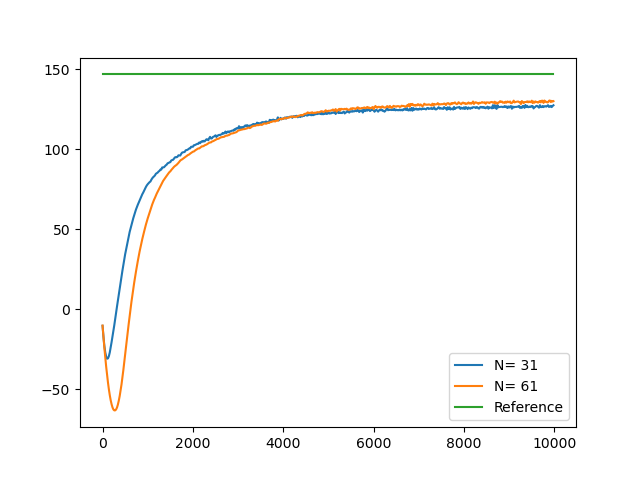}
 \caption*{$p_{samp}=0.64$}
    \end{minipage}
    \caption{Convergence of the actor critic algorithm (function value) for the thermal switching problem depending on the gradient iteration letting $p_{samp}$  or $N$ vary.}
    \label{fig:thermalsamp}
\end{figure}
It is also possible to use the estimated probability and in the simulation, using an Eulerian scheme, the control is selected as the most probable. In the stochastic case, we get a value equal to $146.0$  using $p_{samp}=1$ at the end of the iterations, while the value obtained with $N=61$, $p_{samp}=0.64$ is $139.1$.
\\

We give the control obtained in simulation (taken as the one with the highest probability) using  $p_{samp}=1$ on \cref{fig:thermalControl}.
\begin{figure}[H]
    \centering
    \begin{minipage}[c]{.49\linewidth}
    \includegraphics[width=\linewidth]{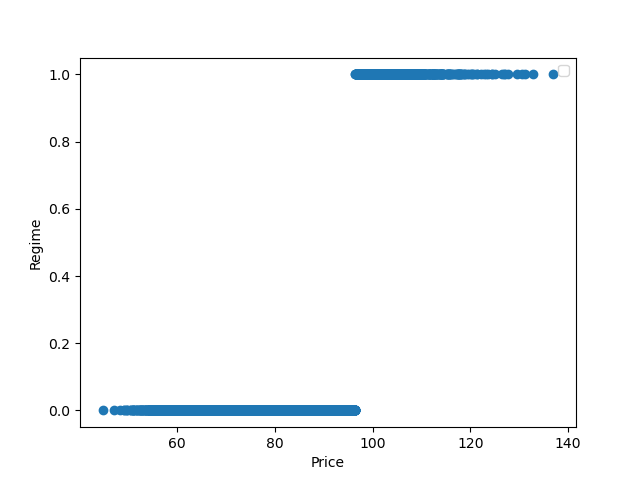}
    \caption*{$t=15$, OFF state (0)}
    \end{minipage}
       \begin{minipage}[c]{.49\linewidth}
    \includegraphics[width=\linewidth]{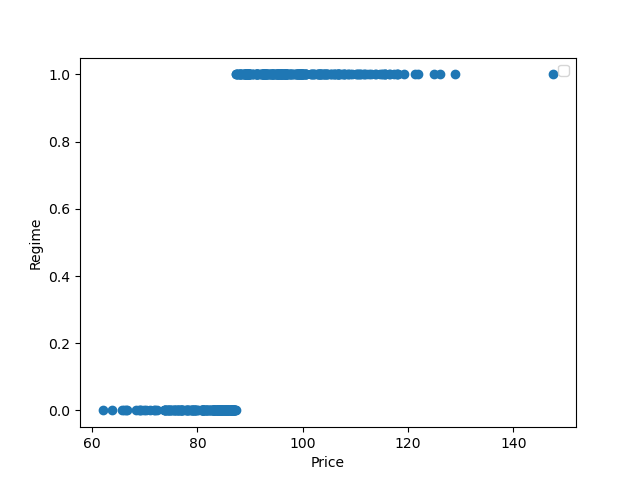}
    \caption*{$t=15$, ON state (1)}
    \end{minipage}
      \begin{minipage}[c]{.49\linewidth}
    \includegraphics[width=\linewidth]{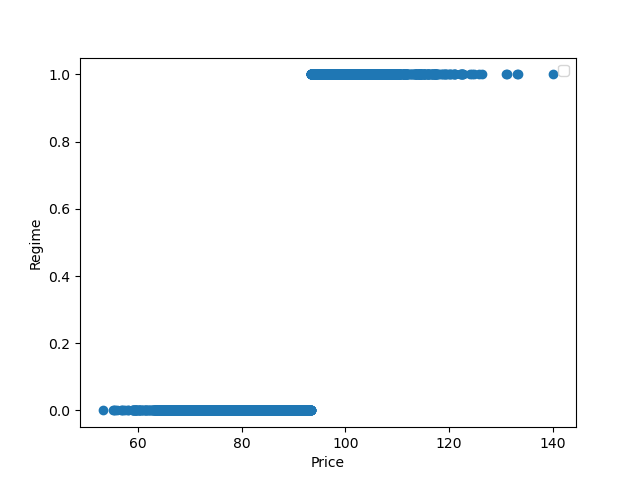}
    \caption*{$t=23$, OFF state (0)}
    \end{minipage}
       \begin{minipage}[c]{.49\linewidth}
    \includegraphics[width=\linewidth]{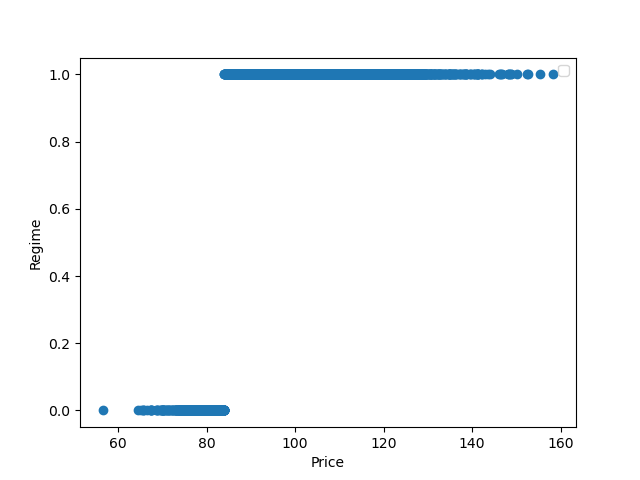}
    \caption*{$t=23$, ON state (1)}
    \end{minipage}
    \caption{Control obtained in the stochastic thermal test case depending the state for two date.}
    \label{fig:thermalControl}
\end{figure}

\subsection{A storage model}

We consider the example of a battery storage valuation formulated as an optimal switching problem, see \cite{carlud10}, \cite{warin2023reservoir}: the manager of the battery aims to price its  real options value by optimizing over 
a finite horizon the dynamic decisions to inject or withdraw power. The inventory process of the battery  is denoted by $(K_t)_t$, and is controlled by a switching control $\alpha$ $=$ $(\alpha_t)_t$: 
\begin{align}
\alpha_t & = \; \alpha_0 \1_{[\tau_0,  \tau_1)} + \sum_{n>0} \xi_n \1_{[\tau_n,\tau_{n+1})}(t), \quad  0 < t \leq T, 
\end{align}
where the random variables $(\xi_n)_n$ denote the sequence of operating regimes valued in $A$ $=$ $\{-1,0,1\}$, representing the decisions to withdraw, do nothing, or inject power.  
At $t=0$, the battery is withdrawing power  so that $\alpha_0= -1$.

The effort of moving from regime $i$ $\in$ $A$ to another regime $j$ $\in$ $A$ incurs a cost $c_{i,j}$ with $c_{i,i}$ $=$ $0$, $c_{i,j}$ $>$ $0$ for $i$ $\neq$ $j$. The inventory is given by:  $K_{t}$ $=$ $\int_0^{t}   \alpha_s ds$,   
while satisfying the physical constraint: $K_{t}$ $\in$ $[0,K_{max}]$, for all $t$ $\in$ $[0,T]$. 
Therefore the number of discrete states  is $k_{max}+1$ where $k_{max}= \frac{K_{max}}{\Delta t}$.
The exogenous price of the electricity is governed by  equation \eqref{eq:price}.
The objective of the manager is to maximize over switching control $\alpha$ the reward functional
\begin{align}
J(\alpha) &= \; \E \Big[ \int_0^T f(S_t,\alpha_t) dt  - \sum_{n} c_{\alpha_{\tau_{n}},\alpha_{\tau_{n+1}}}  \Big], 
\end{align}
with a running profit function
\begin{equation}
f(s,a)  \; = \; \begin{cases} 
 -  s , \quad \mbox{ for } a = 1 \\
 0, \quad   \mbox{ for } a = 0  \\
   s , \quad \mbox{ for } a = -1.
 \end{cases}
\end{equation}
Similarly as  in the previous example:
\begin{itemize}
    \item We model the switching probability at each time step $n$ and at each inventory $k$  by a neural network $\bar\lambda^{\theta_{n,k}}(S_{t_i})$ with parameters $\theta_{n,k}$ and with an output in $[0,1]^9$. When  injection, do nothing and withdraw are allowed (so when $k \in \{ 1, \ldots, k_{max} -1 \}$), for $ (l,m) \in \{-1,0,1\} \times  \{-1, 0, 1\}$, $\bar\lambda^{\theta_{n,k}}_{l,m}(S_{t_n})$ represents the probability  to go from state $l$ at $t_n^{-}$ to state $m$ at $t_n$. When withdraw (respectively injection)   is not allowed therefore when $k=0$ (respectively  $k= k_{max}$),  $\bar\lambda^{\theta_{n,k}}_{l,-1}$  (respectively  $\bar\lambda^{\theta_{n,k}}_{l,1}$) is set to 0. These neural networks use a softmax activation function at the  output layer to satisfy that probabilities are positive with a sum equal to 1.
    \item As for the value function $J$   we use similarly for each time step $n$ and for a level $k$  a neural network  
    $J^{\kappa_{n,k}}(S_{t_n})$ with parameters $\kappa_{n,k}$ and an output in dimension 3 where $J_i^{\kappa_{n,k}}$ is at date $t_n$  and level $k$, the value function in state  $i$ for $i \in \{-1,0,1\}$. As previously we take the convention $J_i^{\kappa_{[N],k}}=0$ for $i \in \{-1,0,1\}$ and each level $k$.
\end{itemize}

Comparing to the thermal switching asset,
\begin{itemize}
    \item the state encompass the inventory level and  we minimize the following loss function to estimate $J$:
\begin{align}
     \sum_{n=0}^{ \bar N-1} \E \Big[\Big| J_{\alpha_{\tau_n}^n}^{\kappa_{[n], K_{\tau_n}^n}}(S_{\tau_{n}}^n)  -  
     \sum_{p=0}^{\bar N-1} f(S_{t_p}^n, \alpha_{t_{p}}^n) \1_{\{t_p \ge \tau_n\}} \Delta t + \sum_{p \ge n} c_{\alpha_{\tau_{p}}^n,\alpha_{\tau_{p+1}}^n}  \1_{\{\tau_{p+1} <T\}} \Big|^2  \Big]
     \label{eq:MLEqB}
 \end{align}
 where at  date $\tau_n$ in the outer summation $S_{\tau_n}^n$ is sampled according the asset law at date $\tau_n$, while the inventory level $K_{\tau_n}^n$ and control applied $\alpha_{\tau_{n}}^n$   are sampled uniformly. $S_{t_p}^n$ is the asset value at date $t_p > \tau_n$ conditionally to its value at date $\tau_n$, and $\alpha_{t_{p}}^n$  the applied control at date $t_p$ starting from $\alpha_{\tau_{n}}^n$ at date $\tau_n$ and obtained sampling the switching probabilities as for the thermal asset. Therefore the flow equation for the inventory level is given for $t_{p} > \tau_n$ by 
 \begin{align}
 \label{eq:flow}
     K_{t_{p+1}}^n=  0 \vee (K_{t_{p}} + \alpha_{t_{p}}^n  \Delta t) \wedge K_{max}
 \end{align}
 \item The gradient function $DW$  is estimated with similar notations as
 \begin{align}
     DW( \theta) & =  \sum_{n=1}^{\bar N-1}  
     \E[\nabla_\theta\log(\bar\lambda^{\theta_{n,K_{t_{n}}}}_{\alpha_{t_{n-1}}^{n-1},\alpha_{t_{n}}^{n-1}}(S_{t_{n}}))  (J^{\kappa_{n,K_{t_{n}}}}_{\alpha_{t_{n}}^{n-1}}(S_{t_{n}}) - J^{\kappa_{n,K_{t_{n}}}}_{\alpha_{t_{n-1}}^{n-1}}(S_{t_n}) 
   -  c_{\alpha_{t_{n-1}}^{n-1},\alpha_{t_{n}}^{n-1}})]
 \end{align}
 where once again for each $n$ in the  loop,  $ S_{t_n}$ is sampled according the asset law at date $t_n$, while the inventory level $K_{t_n}$ at date $t_n$ and the control $\alpha_{t_{n-1}}^{n-1}$ at $t_n^{-}$   are sampled uniformly.
 The control $\alpha_{t_n}^{n-1}$ is sampled from the state $(t_n,K_{t_n},\alpha_{t_{n-1}}^{n-1})$ using the probabilities
 $\bar\lambda^{\theta_{n,K_{t_{n}}}}_{\alpha_{t_{n-1}}^{n-1},\alpha_{t_{n}}^{n-1}}(S_{t_{n}})$.
\end{itemize}

\begin{remark}
    We have to clip values in the flow equation \eqref{eq:flow}. A possible control for  a single  time step may be not admissible if  $\tau_{n+1}-\tau_n>1$. Then if we inject (control $\alpha = 1$), and if    the control is admissible  during one time step  and not for two,  the control is changed to $0$ on the second time step and a switching cost is added. A similar adaptation is carried out in the withdrawal regime.
\end{remark}

As numerical example, we take the following switching costs:
$c_{-1,1}= c_{1,-1} =5$, and for $(i,j)$ not in $ \{(-1,1), (1-1), (-1,-1), (0,0), (1,1) \}$ , $c_{i,j} =3$. We take $K_0 =2$ , $K_{max}=2$ and the reference calculated with the StOpt library using $30$ time steps is $264.3$. 

We use a batch size of $10000$, $\eta_\theta = 0.00015$  and $\eta_\kappa= 0.05$ with ADAM optimizers. 

\begin{figure}[H]
    \centering
    \begin{minipage}[c]{.49\linewidth}
    \includegraphics[width=\linewidth]{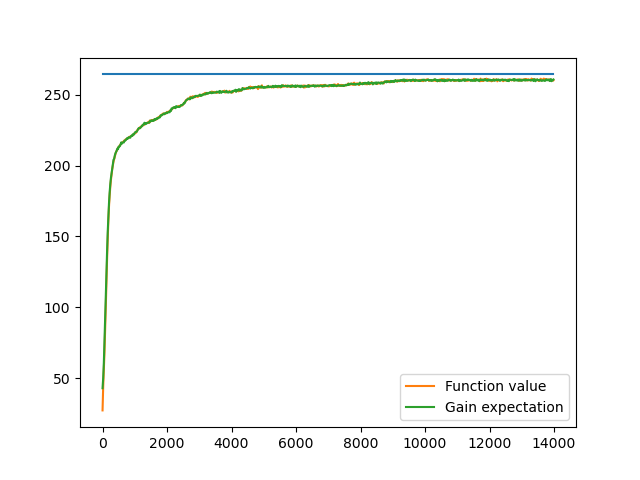}
    \caption*{$p_{samp}=1$, $N=31$.}
        \end{minipage}
       \begin{minipage}[c]{.49\linewidth}
    \includegraphics[width=\linewidth]{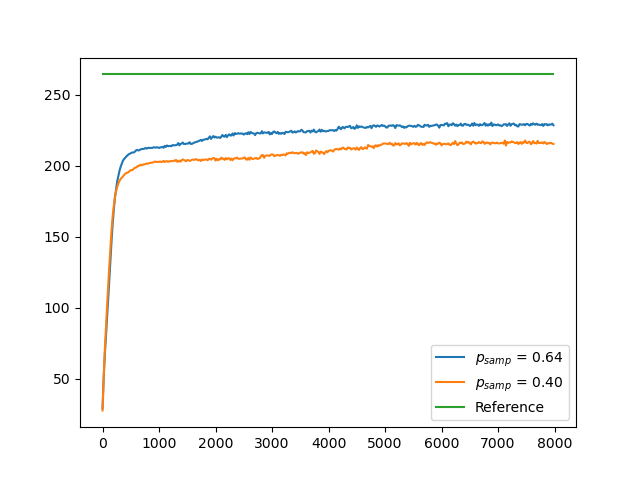}
    \caption*{Results letting $p_{samp}$ vary for $N=31$.}
    \end{minipage}
    \caption{Convergence of the actor critic algorithm for the battery storage case.}
    \label{fig:gasConv}
\end{figure}

On \cref{fig:gasConv}, we observe that taking $p_{samp}=1$ allows to recover almost the exact solution and, as for the thermal asset, the results deteriorate as $p_{samp}$ decreases.

On \cref{fig:gazControlD12,fig:gazControlD25}, we give an example of the control obtained in each regime in simulation.
At each date the control with the highest  probability is taken in simulation.

\begin{figure}[H]
    \centering
     \begin{minipage}[c]{.32\linewidth}
    \includegraphics[width=\linewidth]{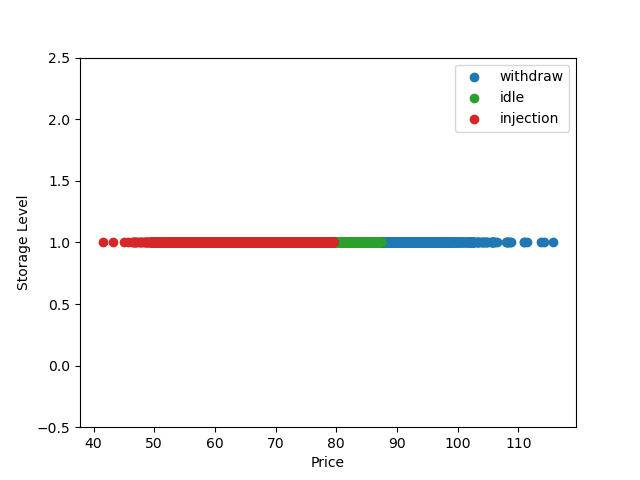}
    \caption*{Injection regime}
    \end{minipage}
       \begin{minipage}[c]{.32\linewidth}
    \includegraphics[width=\linewidth]{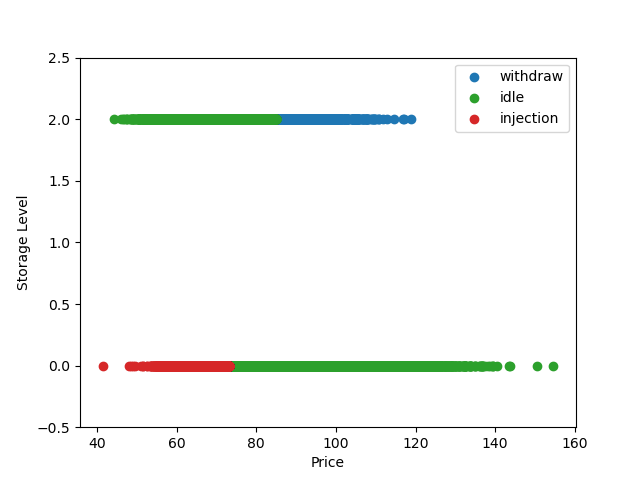}
    \caption*{Idle regime}
    \end{minipage}
      \begin{minipage}[c]{.32\linewidth}
    \includegraphics[width=\linewidth]{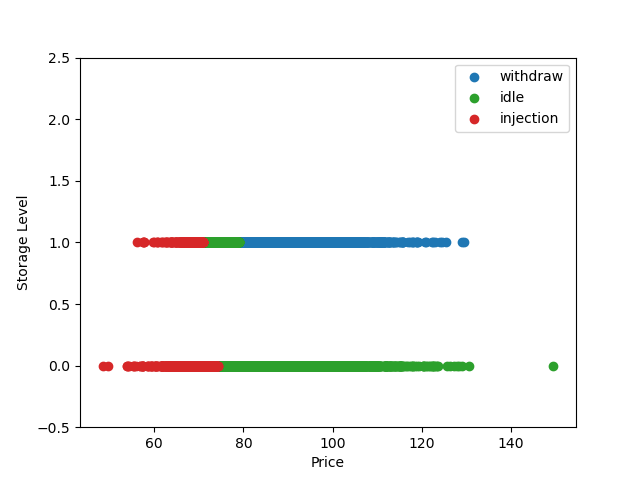}
    \caption*{Withdrawal regime}
    \end{minipage}
    \caption{Control obtained at date 12 in the battery test case depending on the regime.}
    \label{fig:gazControlD12}
\end{figure}
\begin{figure}[H]
    \centering
     \begin{minipage}[c]{.32\linewidth}
    \includegraphics[width=\linewidth]{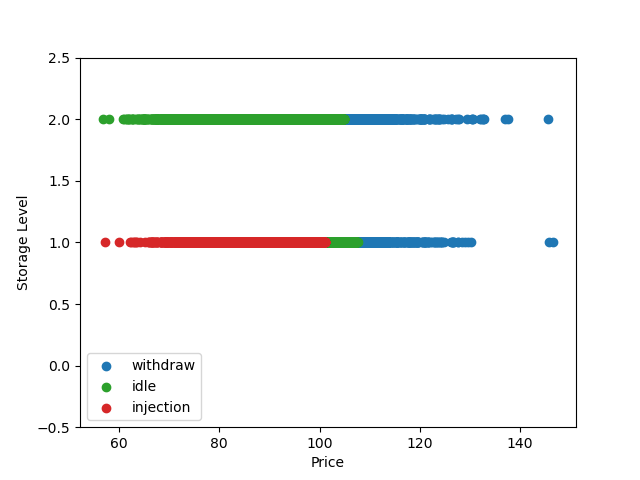}
    \caption*{Injection regime}
    \end{minipage}
       \begin{minipage}[c]{.32\linewidth}
    \includegraphics[width=\linewidth]{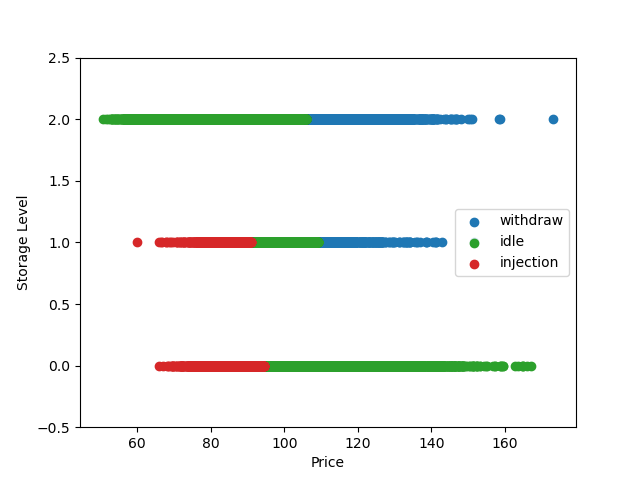}
    \caption*{Idle regime}
    \end{minipage}
      \begin{minipage}[c]{.32\linewidth}
    \includegraphics[width=\linewidth]{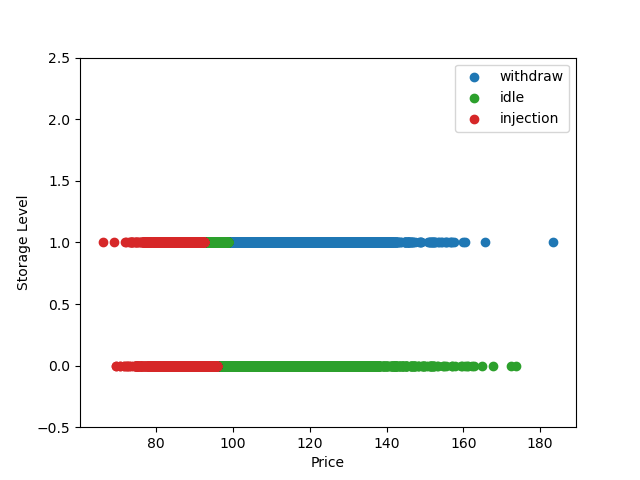}
    \caption*{Withdrawal regime}
    \end{minipage}
    \caption{Control obtained at date 25 in the battery test case depending on the regime.}
    \label{fig:gazControlD25}
\end{figure}

\printbibliography

\vspace{5mm}

\noindent \textbf{Funding}

\noindent The authors declare that no funds, grants, or other supports were received during the preparation of this manuscript.

\vspace{1mm}

\noindent \textbf{Conflict of interest}

\noindent The authors have no relevant or non-financial interests to disclose.

\end{document}